\documentclass[12pt,a4paper,twoside, english]{amsart}
\usepackage[a4paper, left=30mm, right=30mm, top=28mm, bottom=28mm]{geometry}
\usepackage{amsmath,amsfonts,amssymb}
\usepackage{float}
\usepackage{babel}
\usepackage{hyperref}

\usepackage{wrapfig}
\usepackage{booktabs}

\usepackage{stmaryrd}
\usepackage{comment}
\usepackage{caption}
\usepackage{sidecap}
\usepackage{xcolor}

\hypersetup {colorlinks=true, linkcolor=blue, urlcolor=blue, citecolor=blue}

\newcommand\N{\mathbb{N}}
\newcommand\R{\mathbb{R}}
\newcommand\C{\mathbb{C}}
\newcommand\Z{\mathbb{Z}}
\newcommand\eps{\varepsilon}
\DeclareMathOperator{\curl}{\mathrm{curl}}

\newcommand{\dvec}[3]{\left(\begin{matrix}#1\cr #2\cr #3\end{matrix}\right)}

\newcommand\calX{\mathcal{X}}
\newcommand\calY{\mathcal{Y}}
\newcommand\calZ{\mathcal{Z}}

\newcommand{\supp}{\operatorname{supp}}

\newcommand{\del}{\partial}

\newcommand{\one}[1]{{\bf{1}}_{#1}}
\newcommand{\id}{\mathrm{id}}

\newcommand{\loc}{\mathrm{loc}}

\newcommand\fhi{\varphi}

\usepackage{mathtools}
\usepackage{extarrows}
\renewcommand{\div}{\operatorname{div}}

\newcommand{\delh}{\partial_\textrm{hor}\Omega}

\newcommand{\delhU}{\partial_\textrm{hor}U}

\newtheorem{theorem}{Theorem}[section]
\newtheorem{definition}[theorem]{Definition} 
\newtheorem{lemma}[theorem]{Lemma}

\newtheorem{proposition}[theorem]{Proposition}
\newtheorem{remark}[theorem]{Remark}

\def\XXint#1#2#3{{\setbox0=\hbox{$#1{#2#3}{\int}$ }
		\vcenter{\hbox{$#2#3$ }}\kern-.6\wd0}}

\numberwithin{equation}{section}

\allowdisplaybreaks

\begin{document}
	\bibliographystyle{abbrv}
	
	\pagestyle{myheadings} \markboth{On polarization interface
 conditions for time-harmonic Maxwell's
 equations}{B.~Delourme, B.~Schweizer and D.~Wiedemann}
	
	\thispagestyle{empty}

	\begin{center}
		~\vspace{-0.9cm}
		~\vskip3mm {\Large\bf On polarization interface conditions for \\[3mm]
			time-harmonic Maxwell's equations}\\[6mm]
		{\large B.~Delourme\footnotemark[1], B.~Schweizer\footnotemark[2],
			D.~Wiedemann\footnotemark[2]}\\[4mm]
		
	\, 
		
	\end{center}
	
	\footnotetext[1]{Université Sorbonne Paris Nord, Département de
		Mathématiques, Institut Galilée, 99, avenue Jean-Baptiste Clément,
		F-93430 Villetaneuse, delourme$@$math.univ-paris13.fr}
	
	\footnotetext[2]{Technische Universität Dortmund, Fakult\"at f\"ur
		Mathematik, Vogelspothsweg 87, D-44227 Dortmund,
		ben.schweizer$@$tu-dortmund.de, david.wiedemann$@$tu-dortmund.de}
	
	\begin{center}
		\vskip1mm
		\begin{minipage}[c]{0.87\textwidth}
			{\bf Abstract:} We consider the time-harmonic Maxwell's equations
			with a polarization interface condition. The interface condition
			demands that one component of the electric field vanishes at the
			interface and that the corresponding component of the magnetic
			field has no jump across the interface. These conditions have been
			derived in the literature as a homogenization limit for thin wire inclusion. 
			We analyze the limit equations and provide an existence result and a
			Fredholm-alternative.
			
		\end{minipage}\\[5mm]
	\end{center}

	\section{Introduction}
	Polarization filters for electromagnetic waves are interesting for
	many technical applications such as, e.g., LED monitors. The filters
	have the property that waves with a certain polarization can pass the
	filter, waves with the oppposite polarization cannot pass the filter,
	see Figure~\ref {fig:polarization}. Polarization filters are also a
	very interesting mathematical object. Maxwell's equations describe the two
	electromagnetic fields $E$ and $H$ on the two sides of
	the filter, the filter is modelled by some interface conditions. These
	conditions demand continuity for certain components of the fields and
	homogeneous Dirichlet conditions for other components. The
	transmission conditions have been derived recently with mathematical
	rigor. In this article, we are concerned with
	the analysis of the limit model and clarify the well-posedness. We
	treat the time-dependent system; well-posedness of the time-dependent
	system could be concluded with a Fourier transform in time.
	
	\begin{SCfigure}[][ht]
		\begin{minipage}{0.6\linewidth}
			\vspace{-5cm}
			\includegraphics[width=0.45\linewidth]{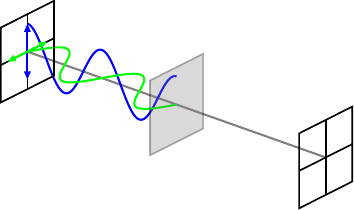}
			\hspace{0.1cm}
			\includegraphics[width=0.45\linewidth]{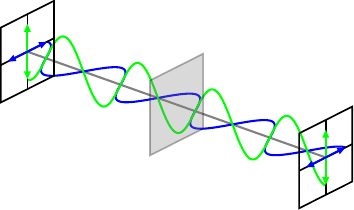}
		\end{minipage}
		\hfill
		\begin{minipage}[ht]{0.38\linewidth}
			\captionsetup{width=0.96\linewidth} {\caption{\small Polarization:
					Electromagnetic waves (electric field in green and magnetic
					field in blue) propagating from the left-hand side and interacting
					with a polarization interface.
					\vspace{1cm} \label{fig:polarization}}}
		\end{minipage}
	\end{SCfigure}
	
	\vspace{-1.2cm} The system of interest is
	\eqref{eq:Maxwell:strong}, it has been derived with homogenization
	techniques in \cite {DelourmeHewett2020} and \cite
	{Schweizer-Wiedemann-Pol-2025}. Let us describe these results briefly,
	using the notation of the latter work. Maxwell's equations are
	studied in a complex domain $\Omega_\eta$, accordingly, the solutions
	$E^\eta$ and $H^\eta$ depend on the parameter $\eta>0$. The domain is
	obtained by removing a small scale structure from an underlying domain
	$\Omega$, an open cuboid in $\R^3$. The cuboid contains a flat
	interface $\Gamma$. A microstructure $\Sigma_\eta$ is defined (in the
	most relevant of three examples) by rods with cross-sectional diameter
	of order $\eta$, rods that are parallel and distributed along $\Gamma$
	with periodicity $\eta$.
	\begin{wrapfigure}{r}{0.4\textwidth} 
		\centering
		\includegraphics[width=0.27\textwidth]{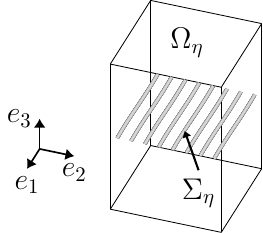}
		\caption{Domains $\Omega_\eta$ and $\Sigma_\eta$.}
		\label{fig:OmegaEta}
	\end{wrapfigure}
	The complex domain is obtained by removing the union of rods from the
	underlying domain: $\Omega_\eta = \Omega\setminus \Sigma_\eta$. The
	geometry is visualized in Figure~\ref{fig:OmegaEta}. Maxwell's
	equations in $\Omega_\eta$ with boundary conditions that model
	perfectly conducting material in the rods provides solutions $E^\eta$,
	$H^\eta$. In \cite {DelourmeHewett2020} and \cite
	{Schweizer-Wiedemann-Pol-2025}, the effective limit system
	\eqref{eq:Maxwell:strong} is derived. In \cite {DelourmeHewett2020},
	asymptotic expansions are used. In \cite
	{Schweizer-Wiedemann-Pol-2025}, for rods parallel to the
	$x_1$-direction, it is shown with oscillating test-functions that,
	whenever there is a weak $L^2$-convergence
	$(E^\eta,H^\eta) \to (E,H)$, the limit $(E,H)$ satisfies system
	\eqref{eq:Maxwell:strong}.
	
	The articles \cite {DelourmeHewett2020} and \cite
	{Schweizer-Wiedemann-Pol-2025} are concerned with the derivation of an
	effective system, not with the analysis of that system. To the best of
	our knowledge, an important question is open: Can we have results on
	existence and uniqueness for the limit system
	\eqref{eq:Maxwell:strong} (in dependence of $\omega$). The article at
	hand answers this question. Let us stress the importance of this step:
	Without this analysis, we cannot be sure that the system
	\eqref{eq:Maxwell:strong} is complete in the sense that all equations
	are found. Additionally, we cannot be sure if
	\eqref{eq:Maxwell:strong} is over-determined (such that it cannot be solved in
	general). With the analysis of the paper at hand we clarify that the
	limit system is well-posed.
	
	On the technical side, we mention that the derivation in \cite
	{Schweizer-Wiedemann-Pol-2025} assumes that the solution sequence
	$(E^\eta,H^\eta)$ is bounded in $L^2$. Since this is not verified,
	\cite {Schweizer-Wiedemann-Pol-2025} does not provide the existence of
	solutions to \eqref{eq:Maxwell:strong}. 
	
	The existence and uniqueness results (or, more precisely, the Fredholm
	alternative) for system~\eqref{eq:Maxwell:strong} turns out to be
	quite tricky. There is a natural approach that we sketch in
	Section~\ref {ssec.variational}: One defines variants of
	$H(\curl)$-spaces that incorporate some of the interface
	conditions. With a compactness result for these $H(\curl)$-spaces one
	derives a Fredholm alternative; this is the classical approach, see
	\cite{Monk2003a}. We did not succeed in showing the desired
	compactness. 
	
	Since we are lacking a compactness result, we must use another
	approach. Our idea is to use an equivalent formulation of the Maxwell system
	with a family of Helmholtz-type systems. The analysis of the latter is quite 
	simple so that we can determine the spectrum of the Maxwell system.

	\subsection{Maxwell equations with a polarization interface condition}
	
	To simplify the setting, we consider cuboids $\Omega\subset \R^3$,
	aligned with the coordinate axes, and an aligned interface. We consider length parameters
	$l_1,l_2, l_3^+, l_3^- >0$ and $l_3 = l_3^+ + l_3^-$ to define the
	intervals
	\begin{align}\label{eq:def:I1-I2-I3}
		I_1 \coloneqq (0,l_1)\,, \ I_2 \coloneqq (0,l_2)\,, \
		I_3 \coloneqq (-l_3,l_3)\,, \
		I_3^+ \coloneqq (0,l_3^+)\,, \ 
		I_3^- \coloneqq (-l_3^-,0)\,.
	\end{align}
	The domain $\Omega$ and the interface $\Gamma$ are
	\begin{equation}\label{eq:def:Omega}
		\Omega \coloneqq I_1 \times I_2 \times I_3 \,,
		\qquad \Gamma \coloneqq I_1 \times I_2 \times \{0\} \,.
	\end{equation}
	The upper part of the domain is
	$\Omega_+ \coloneqq I_1 \times I_2 \times I_3^+$ and the lower part is
	$\Omega_- \coloneqq I_1 \times I_2 \times I_3^-$.
		For the top and the bottom boundary of $\Omega$ we use the notation
	\begin{equation}\label{eq:def:Omega-top-bot}
		\partial_\text{top}\Omega \coloneqq I_1 \times I_2 \times \{l_3^+\}\,,\quad
		\partial_\text{bot}\Omega \coloneqq I_1 \times I_2 \times \{-l_3^-\}\,,\quad
		\delh \coloneqq \partial_\textrm{top}\Omega \cup \partial_\textrm{bot}\Omega\,.
	\end{equation}
	
	\medskip We treat time-harmonic Maxwell's equations, solutions are
	pairs $(E,H)$ with $E, H:\Omega\to \C^3$. The polarization interface
	is given by $\Gamma$. The reader should think of this interface as
	an effective description of long and thin conductive objects that are elongated in the
	first coordinate direction $e_1$, as visualized in
	Figure~\ref{fig:OmegaEta}. The effective behavior of this geometry is given by the interface conditions \eqref
	{eq:Maxwell:strong:3}--\eqref {eq:Maxwell:strong:4}. For notational convenience, we use 
	periodicity boundary conditions along the lateral
	boundaries. In the following system, the frequency $\omega>0$, the permittivity $\eps>0$, the
	permeability $\mu>0$ and source terms $f_h, f_e$ are given. The	
	system has to be solved for $E$ and $H$. 
	
	\vspace*{5mm}
	\hspace{-0.3cm}\begin{minipage}[h]{0.7\textwidth}
		\begin{subequations}\label{eq:Maxwell:strong}
			\begin{align}\label{eq:Maxwell:strong:1}
				\curl E &= i \omega \mu H + f_h && \text{in } \Omega \,,
				\\\label{eq:Maxwell:strong:2}
				\curl H &= - i \omega \eps E + f_e && \text{in } \Omega \setminus \Gamma \,,
				\\\label{eq:Maxwell:strong:3}
				E_1|_\Gamma & = 0 && \text{on } \Gamma \,,
				\\\label{eq:Maxwell:strong:4}
				\llbracket H_1\rrbracket_\Gamma & = 0 && \text{on } \Gamma \,,
				\\\label{eq:Maxwell:strong:5}
				E \times e_3 &= 0 && \text{on } \partial_\text{top}\Omega \cup\partial_\text{bot}\Omega \,,
				\\\label{eq:Maxwell:strong:6}
				x &\mapsto (E,H)(x) && \text{is }x_1\text{- and }x_2\text{-}\text{periodic} \,.
			\end{align}
		\end{subequations}
	\end{minipage}
	\hspace{-1cm}
	\begin{minipage}[h]{0.37\textwidth}
		\centering
		\vspace{0.cm}
		\captionsetup{width=1\linewidth}
		\includegraphics[width=0.7\textwidth]{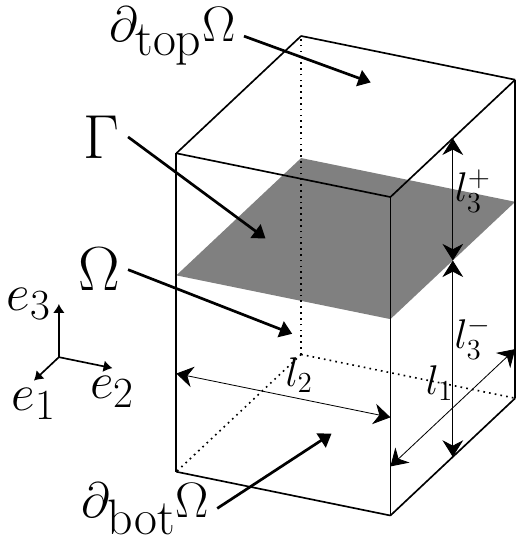}
		\captionof{figure}{Geometry}
		\label{fig:beispiel}
	\end{minipage}
	\smallskip
	
	Equation \eqref{eq:Maxwell:strong:1} is imposed on the entire domain
	$\Omega$; this implies that the tangential components $E_1$ and $E_2$ of $E$ cannot
	jump across $\Gamma$. This implicit interface condition is
	supplemented by the interface conditions \eqref{eq:Maxwell:strong:3}
	and \eqref{eq:Maxwell:strong:4}, which impose that $E_1$ vanishes at
	the interface and that $H_1$ does not jump across the interface. At
	the upper and lower boundaries of $\Omega$,
	\eqref{eq:Maxwell:strong:5} models a perfect conductor. 
	
	The weak solution concept for system \eqref{eq:Maxwell:strong} is made
	precise in Definition~\ref{def:Maxwell:weak}. This definition uses
	function spaces $X \subset H_0(\curl, \Omega)$ and
	$Y\subset H(\curl, \Omega\setminus \Gamma)$, see \eqref{eq:Def:X} and
	\eqref{eq:Def:Y}, respectively. The spaces include boundary
	conditions and interface conditions. Essentially, we will find
	solutions of system \eqref{eq:Maxwell:strong} with $E\in X$ and
	$H\in Y$. Technically, in the definition, the weak solution concept is
	slightly different: it uses $Y$ as a space of test-functions for the
	$E$-equation and $X$ as a space of test-functions for the
	$H$-equation. In the following, when we speak of solutions to
	\eqref{eq:Maxwell:strong}, we always mean weak solutions in the sense
	of Definition~\ref{def:Maxwell:weak}.

	Due to the simple geometry, we can determine the spectrum of the
	Maxwell operator. For lengths $L_1, L_2, L_3 > 0$ we set
	\begin{equation}
		\label{eq:spec-LLL}
		\sigma(L_1, L_2, L_3)
		\coloneqq \left\{ \frac{4 \pi^2}{\eps \mu} \left(\frac{k_1^2}{L_1^2}
		+ \frac{k_2^2}{L_2^2}
		+ \frac{k_3^2}{4L_3^2}\right) \,\middle|\,
		k_1, k_2, k_3\in \N_0 \right\}\,,
	\end{equation}
	with $\N_0 = \{0,1,2,...\}$. These are actually the eigenvalues of the operator
	$-(\eps\mu)^{-1}\Delta$ on the cuboid with side-lengths $L_1, L_2, L_3$
	with periodicity conditions in the first two directions and with a
	Neumann condition in the third direction. The eigenfunctions for this 
	operator $-(\eps\mu)^{-1}\Delta$
	are, when
	$(0,0,0)$ is a corner of the cuboid,
	$u(x_1, x_2, x_3) = \cos(2\pi k_1 x_1/L_1) \cos(2\pi k_2 x_2/L_2)
	\cos(\pi k_3 x_3/L_3)$.
	
	We will see that the spectrum of the Maxwell operator in the geometry
	of \eqref{eq:def:Omega} is given by
	\begin{equation}
		\label{eq:spec-M}
		\sigma_M
		\coloneqq \sigma(l_1, l_2, l_3) \cup \sigma(l_1, l_2, l_3^+)
		\cup \sigma(l_1, l_2, l_3^-)\,.
	\end{equation}
	It is the union of the spectra of $(\eps\mu)^{-1}\Delta$ on the total
	domain, the upper part and the lower part. Our main result determines
	the spectrum of the Maxwell system.
	
	\begin{theorem}[Spectrum of the Maxwell system with a polarization
		interface]
		\label{thm:spectrum} Let $\Omega$ and $\Gamma$ be as in
		\eqref{eq:def:Omega}, we consider the Maxwell system
		\eqref{eq:Maxwell:strong} with parameters $\eps, \mu >0$. The
		spectrum of the Maxwell system \eqref{eq:Maxwell:strong} is
		$\sigma_M$ of \eqref {eq:spec-M} in the following sense: (i) For
		$\omega^2 \notin \sigma_M$, system \eqref{eq:Maxwell:strong} has a
		unique solution for arbitrary $(f_h,f_e) \in L^2(\Omega,
		\C^3)^2$. (ii) For $\omega^2 \in \sigma_M$, system
		\eqref{eq:Maxwell:strong} has a non-trivial solution for
		$(f_h,f_e) = 0$.
	\end{theorem}
	
	Theorem \ref {thm:spectrum} implies the statement of the subsequent
	remark, which is less precise in the description of the spectrum.
	Remark \ref {rem:fredholm} formulates a statement that we expect to be
	valid not only in the cuboid geometry with a flat interface, but also
	for more general domains $\Omega$ and more general interfaces
	$\Gamma$.
	
	\begin{remark}[Spectrum and Fredholm alternative]
		\label{rem:fredholm}
		For parameters $\eps, \mu> 0$ and in the geometry of
		\eqref{eq:def:Omega}, there holds for some discrete set
		$\sigma\subset \R$: For every $0\neq\omega^2\notin \sigma$, there exists
		a unique solution $(E,H) \in X \times Y$ of
		\eqref{eq:Maxwell:strong} for every
		$f_h, f_e \in L^2(\Omega, \C^3)$. For every $0\neq\omega^2\in \sigma$,
		the homogeneous system \eqref{eq:Maxwell:strong} has a non-trivial
		solution.
	\end{remark}
	
	In the above statements, it is no loss of generality to assume that
	the right-hand sides $f_h, f_e \in L^2(\Omega, \C^3)$ are divergence
	free in the distributional sense. Indeed, general $f_h, f_e$ can be
	decomposed with a Helmholtz decomposition to reduce the Maxwell system
	to divergence-free right-hand sides. We sketch the well-known argument
	in Lemma~\ref{lem:Helmholtz:Decomposition:Data} and
	Remark~\ref{rem:Helmholtz:Decomposition:Data:general}. 
	
	\subsection{On a variational formulation}
	\label{ssec.variational}
	
	The Maxwell system \eqref{eq:Maxwell:strong} possesses a natural
	variational formulation. Let us describe this formulation even though
	our proof on existence and uniqueness follows a different route and
	uses Helmholtz equations.
	
	\smallskip We start with the ``coercive Maxwell system'', which is
	formally obtained by setting $\omega = i$. For notational
	convenience, we assume in this overview $f_h = 0$. Using $E$ as the
	only unknown, the coercive Maxwell system reads
	$\curl(\curl\, E) + \eps\mu\, E = -\mu f_e$. We use the function
	space $X$, defined in \eqref{eq:Def:X}, which consists of functions in
	$H(\curl, \Omega)$ which are periodic in the first and second
	direction, have vanishing tangential components at $\delh$ and a
	vanishing first component at $\Gamma$. We define the sesquilinear
	form $B \colon X \times X \to \C$,
	\begin{equation}
		\label{eq:var-form-B}
		B(E,\phi) \coloneqq \int_\Omega \left\{ \curl E\cdot \curl \bar\phi
		+ \eps\mu\,E\cdot\bar\phi\right\}\,.
	\end{equation}
	The form $B$ is coercive on $X$ since $B(E,E)$ controls the
	$L^2$-norm of $E$ and the $L^2$-norm of $\curl E$. By the Lax--Milgram lemma, for
	every $f_e\in L^2(\Omega)$, we can solve uniquely the problem
	$B(E,\phi) = -\mu \int f_e\cdot\bar\phi$\ $\forall \phi\in X$, which
	implies $\curl(\curl\, E) + \eps\mu\, E = -\mu f_e$ in
	$\Omega\setminus \Gamma$. By setting $H \coloneqq -\mu^{-1}\curl E$, due to
	our choice of function spaces, we see that all equations of
	\eqref{eq:Maxwell:strong} are satisfied with $\omega = i$ and $f_h=0$
	--- the only open point being \eqref{eq:Maxwell:strong:4} concerning the jump of $H_1$ across $\Gamma$.
	
	The fact that we have solved $B(E,\phi) = -\mu\int f_e\cdot\bar\phi$
	implies that
	\begin{equation*}
		- \int_\Omega H\cdot \curl \bar\phi
		+ \int_\Omega \eps\,E\cdot\bar\phi = - \int_\Omega f_e\cdot \bar\phi
	\end{equation*}
	holds for all $\phi\in X$. This encodes not only
	\eqref{eq:Maxwell:strong:2}, but also, since $\curl\,H|_{\Omega_\pm}$
	is an $L^2(\Omega_\pm)$-function, that the interface integral over the
	jump vanishes,
	\begin{equation*}
		0 = \int_\Gamma (e_3\times \llbracket H\rrbracket_\Gamma)\cdot \bar\phi
		= \int_\Gamma \llbracket H_1\rrbracket_\Gamma\ \bar\phi_2
	\end{equation*}
	for all $\phi\in X$. Since $\phi_2$ can take arbitrary values on the
	inner interface $\Gamma$, the function $H_1$ cannot have a jump across
	$\Gamma$ (in the sense of traces). With this argument, also
	\eqref{eq:Maxwell:strong:4} is verified. We conclude that, for the ``coercive Maxwell system'', 
	the weak form can be solved uniquely and it encodes all equations. 
	
	\smallskip The next goal is to transfer this existence and uniqueness
	result for $\omega = i$ to arbitrary $\omega\in \R$, or, more
	precisely, to all $\omega\in \R\setminus \{0\}$ except for the countable set of
	eigenvalues. The underlying idea is to perform the following Steps:
	(i) Repeat the arguments on some smaller space $X_{\div} \subset X$,
	which is defined by introducing an additional condition on the
	divergence of the functions $E$. (ii) Show that the space $X_{\div}$
	is a compact subset of $L^2(\Omega)$. (iii) The difference between the
	original Maxwell system and the coercive system consists in the
	multiplication operator $(\omega^2 + 1)\,\id$. This operator is a bounded
	operator on $L^2(\Omega)$ and, hence, defines a compact perturbation
	of the system by the compactness of (ii). The conclusion of (i)--(iii) 
	is that the Maxwell system defines a Fredholm operator.
	
	Without the interface $\Gamma$, this program works. The space
	$X_{\div}\subset X$ is defined by imposing $\nabla\cdot E = 0$ in
	the distributional sense. The compactness is classical, see, e.g.,
	\cite{Monk2003a}. We obtain that the Maxwell system is Fredholm.
	
	With the interface $\Gamma$, we did not succeed to perform a
 proof along these lines. The main obstacle regards Steps (i)
 and (ii): We did not find a suitable divergence condition that
 allows to derive the compactness of $X_{\div}$ and to repeat
 the remaining arguments.
	
	\subsection{Literature}
	
	This contribution is related to compactness of
 $H(\curl)$-spaces, for the classical theory see
 \cite{Monk2003a} and \cite{Kirsch-Buch-2015}, extensions to
 mixed boundary conditions are made in
 \cite{Bauer-Pauly-S-2016}, other methods are used
 \cite {Ciarlet-2005}, a very general treatment on the
 related Helmholtz decompositions is given in \cite
 {ABDG-1998}. We also refer to the references in these works
 and to a short summary given in \cite
 {Schweizer-Wiedemann-Max-Ex-2024}. A Fredholm alternative for
 Maxwell's equations in a quasiperiodic setting is the basis of
 the analysis in \cite{Kirsch-Schweizer-ARMA2025}. We emphasize
 that all of these works do not treat polarization interfaces.
	
	The homogenization of bulk heterogeneities is very classical, often
	treated with two-scale convergence, see \cite{MR1185639}. Interesting
	limit equations can occur when resonances are exploited, see
	\cite{MR2576911-Bouch} and \cite {Bouchitte-Schweizer-Max-2010}.
 In
	\cite{Ohlberger-Schweizer-Urban-Verfuehrt-2020}, a closely related
	setting is investigated: Maxwell's equations in a domain with thin
	structures, possibly wires; in contrast to the present article, the
	inclusions are filling a subdomain and a bulk homogenization is
	performed. The work contains numerical aspects, more in this direction
	can be found in \cite{MR4096127} and the references therein.
	
	Small inclusions along a hypersurface require different methods, but
	they are also a classical topic, see \cite{MR1493040}. A detailed
	study of inclusions along a hypersurface is given for Helmholtz
	equations in \cite{DelourmeHaddarJoly-2012} and
	\cite{Schweizer-Neumann-2020}. These results concern the critical
	scaling and treat $\eta^0$ and $\eta^1$ approximations of
	solutions. In this sense, those results are much more detailed than
	those of \cite {DelourmeHewett2020} and
	\cite{Schweizer-Wiedemann-Pol-2025}, which are the homogenization
	counterpart of the work at hand.

	\subsection{Reformulation as a Helmholtz-type system}
	We consider generalized di\-ver\-gence-free right-hand sides $f_e$ and $f_h$ in the sense that $f_e \in X_0^\perp$ and $f_h \in Y_0^\perp$ for $X_0$ and $Y_0$ given by \eqref{eq:def:X0-Y0}. These properties imply, written in a strong form,
	\begin{align}\label{eq:strongdivFree}
		\div f_e = 0 \ \text{ in } \Omega\setminus \Gamma\,, \qquad \div f_h = 0 \ \text{ in } \Omega \,, \qquad (f_h)_3|_{\delh} =0 \,.
	\end{align}
	Our results are based on the observation that for such right-hand sides the Maxwell system
	\eqref{eq:Maxwell:strong} is equivalent to the Helmholtz-type system \eqref {eq:Helmholtz:strong}
	below. We only have to demand a non-resonance condition on the
	frequency, namely $\omega^2 \notin \sigma(l_1) \subset \sigma_M$ for
	\begin{equation*}
		\sigma(L_1) \coloneqq \left\{ \frac{4 \pi^2}{\eps \mu} \frac{k_1^2}{L_1^2}
		 \,\middle|\,	k_1\in \N_0 \right\}\,.
	\end{equation*}
	
	The six equations of \eqref {eq:Helmholtz:strong} provide a strong
	formulation of the Helmholtz-type system. We emphasize that this
	strong formulation cannot be used for right-hand sides $f_h$ and $f_e$
	of class $L^2(\Omega)$; this regards, in particular, the boundary
	condition \eqref {eq:Helmholz:BCH}, since, in general, $f_e$ does not
	have a trace. It is therefore necessary to understand system \eqref
	{eq:Helmholtz:strong} in the weak sense, provided in Definition
	\ref{def:Helmholtz:weak}. The strong formulation is given here only for
	convenience of the reader.
	\begin{subequations}\label{eq:Helmholtz:strong}
		\begin{align}\label{eq:Helmholtz:strong:E1}
			\Delta E_1 + \omega^2 \eps \mu E_1
			&= - i \omega \mu (f_e)_1 - (\curl f_h)_1\,,
			\\
			\label{eq:Helmholtz:strong:H1}
			\Delta H_1 + \omega^2 \eps \mu H_1
			&= i \omega \eps (f_h)_1 - (\curl f_e)_1 \,,
			\\
			\label{eq:Helmholtz:strong:E2}
			\partial_1^2 E_2 + \omega^2 \eps \mu E_2
			&= \partial_2 \partial_1 E_1 + i \omega \mu \partial_3 H_1 + \partial_1(f_h)_3 - i \omega \mu (f_e)_2\,,
			\\
			\label{eq:Helmholtz:strong:H2}
			\partial_1^2 H_2 + \omega^2 \eps \mu H_2
			&= \partial_2 \partial_1 H_1 - i \omega \eps \partial_3 E_1 - \partial_1 (f_e)_3 + i \omega \eps (f_h)_2 \,,
			\\
			\label{eq:Helmholtz:strong:E3}
			\partial_1^2 E_3 + \omega^2 \eps \mu E_3
			&= \partial_3 \partial_1 E_1 - i \omega \mu \partial_2 H_1 - \partial_1 (f_h)_2 - i \omega \mu (f_e)_3\,,
			\\
			\label{eq:Helmholtz:strong:H3}
			\partial_1^2 H_3 + \omega^2 \eps \mu H_3
			&= \partial_3 \partial_1 H_1 + i \omega \eps \partial_2 E_1 - \partial_1 (f_e)_2 + i \omega \eps (f_h)_3\,.
		\end{align}
	\end{subequations}
	The equations are complemented with the boundary and interface
	conditions
	\begin{subequations}\label{eq:Helmholz:BC}
		\begin{align}\label{eq:Helmholz:BCE}
			E_1 &= 0
			&&\text{on } \Gamma\cup \partial_\text{top}\Omega \cup \partial_\text{bot}\Omega\,,
			\\\label{eq:Helmholz:BCH}
			\partial_3 H_1
			& = (f_e)_2
			&&\text{on } \partial_\text{top}\Omega \cup \partial_\text{bot}\Omega\,,
			\\\label{eq:Helmholz:BCPer}
			x &\mapsto (E,H)(x)
			&&\text{is } x_1\text{- and }x_2\text{-}\text{periodic} \,.
		\end{align}
	\end{subequations}
	We remark that \eqref {eq:Helmholz:BCH} is the natural boundary
	condition when \eqref {eq:Helmholtz:strong:H1} is written in a weak
	form.

	System \eqref {eq:Helmholtz:strong}, complemented with boundary
	conditions \eqref{eq:Helmholz:BC}, for $\omega^2\notin \sigma_M$, can
	be solved in three steps: In the first step, $E_1$ is found as the
	solution to the Helmholtz problem \eqref{eq:Helmholtz:strong:E1} with
	the boundary conditions \eqref{eq:Helmholz:BCE} and
	\eqref{eq:Helmholz:BCPer}; one has to solve two uncoupled Helmholtz
	problems, posed in $\Omega_+$ and $\Omega_-$. In the second step,
	$H_1$ is found as the solution to the Helmholtz problem
	\eqref{eq:Helmholtz:strong:H1} with the boundary conditions
	\eqref{eq:Helmholz:BCH} and \eqref{eq:Helmholz:BCPer}; one has to
	solve one Helmholtz problems in $\Omega$. In the third step, the four
	unknowns $E_2$, $E_3$, $H_2$, $H_3$ are found by solving, for almost
	every $(x_2, x_3)\in I_2\times I_3$, one dimensional Helmholtz problem
	on $I_1$ with periodicity boundary conditions.
	
	In the above procedure, steps one and two can be interchanged. Once
	$E_1$ and $H_1$ are found, the last four equation can be solved
	independently of each other.
	
	Based on the three step procedure, it is easy to show that the
	Helmholtz-type system \eqref {eq:Helmholtz:strong} can be solved
	uniquely for every $\omega^2\notin \sigma_M$. The existence statement
	of Theorem~\ref {thm:spectrum} then follows from the fact that every
	solution to the Helmholtz-type system is a solution to the Maxwell
	system if $\omega^2 \notin \sigma(l_1) \subset \sigma_M$. This latter
	fact is derived in Lemma~\ref {lem:Helmholtz-implies-Maxwell}.

	The uniqueness statement of part (i) of Theorem~\ref
	{thm:spectrum} is a consequence of the fact that the two systems of
	equations are equivalent (we show this even for all $\omega^2 \notin \sigma(l_1)$). 
	The uniqueness statement for solutions of the Helmholtz system,
	which holds for $\omega^2 \notin \sigma_M$, yields a uniqueness
	statement for the solutions of the Maxwell system under the same frequency assumption. 
	
	In all of these arguments,
	one has to be careful in the choice of weak solution
	concepts. Furthermore, one must reduce the problem to right-hand
	sides $f_h, f_e \in L^2(\Omega, \C^3)$ that are divergence-free; more
	precisely, we demand $f_h \in Y_0^\perp$ and $f_e \in X_0^\perp$, for
	the definitions of the spaces see \eqref{eq:def:X0-Y0} below. As
	announced, a Helmholtz decomposition argument provides that it is no
	restriction to consider only divergence-free right-hand sides, see
	Lemma \ref{lem:Helmholtz:Decomposition:Data}.
	
	The following proposition gives a precise description of the fact
	that, loosely speaking, systems \eqref{eq:Maxwell:strong} and
	\eqref{eq:Helmholtz:strong} are equivalent.
	
	\begin{proposition}[Equivalence of the Maxwell and the Helmholtz-type
		system]
		\label{prop:Equi:Maxwell-Helmholtz}
		Let the geometry and the coefficients be as above. Then, the Maxwell
		system \eqref{eq:Maxwell:strong} and the Helmholtz-type system are
		equivalent in the following sense:
		
		\smallskip
		(i) For $\omega^2 \notin \sigma(l_1)$ and a right-hand side given by
		$(f_e, f_h) \in X_0^\perp \times Y_0^\perp$, $(E,H)$ is a weak
		solution to the Maxwell system \eqref{eq:Maxwell:strong} in the
		sense of Definition~\ref{def:Maxwell:weak} if, and only if, it is a
		weak solution to the Helmholtz-type system
		\eqref{eq:Helmholtz:strong} in the sense of
		Definition~\ref{def:Helmholtz:weak}.
		
		\smallskip
		(ii) For $\omega^2 \in \sigma(l_1)$ and $f_e=f_h=0$, there exists a
		non-trivial weak solution $(E_M,H_M)$ to the Maxwell system
		\eqref{eq:Maxwell:strong} and a non-trivial weak solution
		$(E_H,H_H)$ to the Helmholtz-type system
		\eqref{eq:Helmholtz:strong}. Moreover, every solution $(E_M,H_M)$ to
		the Maxwell system is a solution to the Helmholtz-type system. It is
		not true that every solution $(E_H,H_H)$ to the Helmholtz-type
		system is a solution to the Maxwell system.
	\end{proposition}

	\begin{proof} We show part (i) of the proposition in two lemmas.
		Lemma~\ref{lem:Maxwell-implies-Helmholtz} provides that every
		solution of the Maxwell system is a solution of the Helmholtz
		system. Lemma~\ref {lem:Helmholtz-implies-Maxwell}, which is only
		valid for $\omega^2\notin \sigma(l_1)$, provides that every solution
		of the Helmholtz system is a solution of the Maxwell system.
		
		Regarding part (ii), we note that $\sigma(l_1) \subset \sigma_M$. Thus, we
		can apply Lemma \ref{lem:eigenfct-Max}, which provides explicit
		solutions for the homogeneous Maxwell system, for
		$\omega^2\in \sigma_M$. Since
		Lemma~\ref{lem:Maxwell-implies-Helmholtz} is valid for arbitrary
		$\omega^2>0$, these solutions are also solution to the homogeneous
		Helmholtz-type system.
		
		Regarding the last sentence, we give an example in Remark \ref
		{rem:on-equivalence}.
	\end{proof}

	\begin{remark}[On the failure of equivalence for special frequencies]
		\label{rem:on-equivalence}
		In the proof of Proposition \ref {prop:Equi:Maxwell-Helmholtz} (ii), we did not need the equivalence
		of the systems; instead, we constructed solutions for the
		homogeneous Maxwell-type system in Lemma~\ref{lem:eigenfct-Max}.
		Lemma~\ref {lem:Maxwell-implies-Helmholtz} guarantees that they
		are also solutions to the Helmholtz-type system.
		
		In the case $\omega^2 \in \sigma(l_1)$, we consider the following example: For
		$\omega^2 = 4 \pi^2 k_1^2/(\eps \mu l_1^2)$ with some $0 < k_1 \in \N$,
		we set $E_1 = E_3 = 0$, $E_2 = \sin(2 \pi k_1 x_1 /l_1)$ and
		$H = 0$. Then $(E, H)$ is a solution to the homogeneous
		Helmholtz-type system. On the other hand, because of
		$\curl E = 2 \pi k_1 /l_1 \cos(2 \pi k_1 x_1 /l_1) e_3 \neq 0 = i
		\omega \mu H$ the function $(E,H)$ is not a solution to the
		homogeneous Maxwell-type system.
	\end{remark}

	\subsection{Formal equivalence of the Maxwell system and the Helmholtz-type system}
	\label{ssec.Helmholtz-type-system} In this motivational section, we
	present calculations for smooth solutions. We consider only smooth and
	divergence-free source functions that satisfy \eqref{eq:strongdivFree}.
	
	\subsubsection*{From the Maxwell system to the Helmholtz-type system}
	Since the source terms have vanishing divergence, the same is true
	for $E$ in $\Omega\setminus \Gamma$ and for $H$ in $\Omega$, see the first two
	equations of \eqref{eq:Maxwell:strong}. 
	Taking the curl of these two equations and
	using the identity $-\Delta = \curl\curl - \nabla\div$, we obtain the Helmholtz equations
	\eqref{eq:Helmholtz:strong:E1} and
	\eqref{eq:Helmholtz:strong:H1} for the components $E_1$ and $H_1$.
	Since \eqref{eq:Maxwell:strong:2} is imposed only in $\Omega \setminus \Gamma$, the Helmholtz equations hold only on $\Omega\setminus \Gamma$. Let us sketch why \eqref{eq:Helmholtz:strong:H1} holds also across the interface. Firstly, we note that \eqref{eq:Maxwell:strong:1} is posed on $\Omega$ so that
	\begin{align}\label{eq:mot:neumanninterface}
		\llbracket H_3 \rrbracket_\Gamma = \frac{1}{i\omega \mu} (-\llbracket \partial_2 E_1 \rrbracket_\Gamma + \llbracket\partial_1 E_2\rrbracket_\Gamma -\llbracket (f_h)_3 \rrbracket_\Gamma) = 0\,.
	\end{align} 
	Therefore, the third component of $H$ has no jump.
	From \eqref{eq:Maxwell:strong:1} we therefore obtain, since $E_2$ has no jump,
	\begin{align*}
		\llbracket \partial_3 H_1 \rrbracket_\Gamma = \llbracket \partial_1 H_3 \rrbracket_\Gamma + \llbracket (f_e)_2 \rrbracket_\Gamma = 0\,.
	\end{align*}
	Relation \eqref{eq:Maxwell:strong:4} implies that $H_1$ has no jump across the interface. Since neither $H_1$ nor $\del_3 H_1$ have a jump across the interface, \eqref{eq:Helmholtz:strong:H1} holds in all of $\Omega$.

	The boundary condition \eqref{eq:Helmholz:BCE} was demanded with \eqref{eq:Maxwell:strong:3} and \eqref{eq:Maxwell:strong:5}. The component $H_3$ vanishes on the horizontal boundaries by \eqref{eq:Maxwell:strong:1} and the assumption that $(f_h)_3|_{\delh} = 0$. 
	The Neumann boundary condition \eqref{eq:Helmholz:BCH} follows from \eqref{eq:Maxwell:strong:2}.

	Regarding the other equations, we present the calculation for
	\eqref{eq:Helmholtz:strong:E2}. We look at two components of
	\eqref{eq:Maxwell:strong}, namely
	\begin{align*}
		\del_1 E_2 - \del_2 E_1 &= i\omega\mu H_3 + (f_h)_3\,,\\
		\del_3 H_1 - \del_1 H_3 &= -i\omega\eps E_2 + (f_e)_2\,.
	\end{align*}
	Applying $\del_1$ to the first equation and inserting $\del_1 H_3$
	from the second equation, we find
	\begin{align*}
		\del_1^2 E_2 - \del_1\del_2 E_1
		&= i\omega\mu \del_1 H_3 + \del_1 (f_h)_3\\
		& = i\omega\mu [\del_3 H_1 + i\omega\eps E_2 - (f_e)_2] + \del_1 (f_h)_3\\
		& = -\omega^2 \mu \eps E_2 +
		i\omega\mu [\del_3 H_1 - (f_e)_2] + \del_1 (f_h)_3\,.
	\end{align*}
	This is exactly \eqref{eq:Helmholtz:strong:E2}. The other equations
	are derived accordingly.

	\subsubsection*{From the Helmholtz-type system to the Maxwell system}
	As in the above implication, we use different calculations for the first components and for the other components. It is common to all calculations that we use the invertibility of the operator $(\partial_1^2 + \eps\mu \omega^2)$ on the interval $(0,l_1)$.
	We present the calculations for the first two components of \eqref{eq:Maxwell:strong:2}, the other relations are derived analogously.
	
	To show the first component of \eqref{eq:Maxwell:strong:2}, we multiply \eqref{eq:Helmholtz:strong:E1} by $-i\omega\eps$, take the derivative of \eqref{eq:Helmholtz:strong:H2} with respect to the third component and the derivative of \eqref{eq:Helmholtz:strong:H3} with respect to the second component.
	Summing the first two resulting equations and subtracting the third yields:
	\begin{align*}
		\begin{aligned}
			(\partial_1^2 + \eps \mu\omega^2) (\partial_3 H_2 -\partial_2 H_3)
			&= -(\Delta + \eps \mu\omega^2)(- i \omega \eps E_1)
			- \eps \mu \omega^2(f_e)_1 + i\omega \eps (\curl f_h)_1
			\\
			&\quad +
			\partial_3 ( \partial_2 \partial_1 H_1 - i \omega \eps \partial_3 E_1 - \partial_1 (f_e)_3 + i \omega \eps (f_h)_2)
			\\
			&\quad -
			\partial_2 (\partial_3 \partial_1 H_1 + i \omega \eps \partial_2 E_1 - \partial_1 (f_e)_2 + i \omega \eps (f_h)_3 )\\
			&= (\partial_1^2 + \eps \mu\omega^2)(i \omega \eps E_1 -(f_e)_1)\,,
		\end{aligned}
	\end{align*}
	where we used $\div f_e=0$.
	Inverting the operator $(\partial_1^2 + \eps \mu\omega^2)$ yields the first component of \eqref{eq:Maxwell:strong:2}.
	
	To show the second component of \eqref{eq:Maxwell:strong:2},
	we multiply \eqref{eq:Helmholtz:strong:E2} by $i\omega\eps$ and take the derivative of \eqref{eq:Helmholtz:strong:H3} with respect to the first component. Subtracting the second relation from the first yields
	\begin{align*}
		(\partial_1^2 + \eps \mu\omega^2) (\partial_3 H_1 -\partial_1 H_3) = (\partial_1^2 + \eps \mu\omega^2)(- i \omega \eps E_2 + (f_e)_2)\,.
	\end{align*}
	Inverting the operator $(\partial_1^2 + \eps \mu\omega^2)$ yields the second component of \eqref{eq:Maxwell:strong:2}. 
	
	It remains to verify the interface and boundary conditions \eqref{eq:Maxwell:strong:3}--\eqref{eq:Maxwell:strong:6} and, in order to obtain that 
	\eqref{eq:Maxwell:strong:1} holds across $\Gamma$, that $E_1$ and $E_2$ have no jump across $\Gamma$.
	Equation \eqref{eq:Helmholz:BCE} shows that $E_1$ vanishes at $\Gamma$, which gives \eqref{eq:Maxwell:strong:3} and shows that $E_1$ is continuous across $\Gamma$.
	Equation \eqref{eq:Helmholtz:strong:H1} implies that both $H_1$ and $\partial_3 H_1$ have no jump across $\Gamma$, which shows, in particular, \eqref{eq:Maxwell:strong:4}.
	Since $\partial_3 H_1$ does not jump across $\Gamma$, the whole right-hand side of \eqref{eq:Helmholtz:strong:E2} and, thus the left-hand side does not jump across $\Gamma$. By inverting the operator $(\partial_1^2 + \eps \mu\omega^2)$, we obtain the continuity of $E_2$ across $\Gamma$.
	
	In order to derive \eqref{eq:Maxwell:strong:5},
	it remains to show that $E_2$ vanishes at $\delh$.
	For this, we note that the right-hand side of \eqref{eq:Helmholtz:strong:E2} vanishes at $\delh$ since $E_1$ vanishes at $\delh$ and the remaining terms cancel by the boundary condition \eqref{eq:Helmholz:BCH} and the assumption $(f_h)_3|_{\delh} = 0$ given by \eqref{eq:strongdivFree}. Inverting the operator $(\partial_1^2 + \eps \mu\omega^2)$ in \eqref{eq:Helmholtz:strong:E2} gives $E_2|_{\delh} =0$.

	\subsection{Eigenfunctions for the Maxwell system}
	
	The next lemma provides non-trivial solutions of the homogeneous Maxwell system for
	$\omega^2 \in \sigma_M$.
	\begin{lemma}[Eigenfunctions of the Maxwell system]
		\label{lem:eigenfct-Max}
		For frequencies $\omega>0$ with $\omega^2 \in \sigma_M$, the
		homogeneous Maxwell system has a non-trivial solution. We provide an
		explicit solution by distinguishing two cases. In the case
		$\omega^2 \in \sigma(l_1,l_2,l_3)$, we find eigenfunctions that do
		not vanish along $\Gamma$. In the case
		$\omega^2 \in \sigma_M \setminus\sigma(l_1,l_2,l_3)$, we find
		eigenfunctions that live in the upper or in the lower domain.
		
		\smallskip Case $\omega^2 \in \sigma(l_1,l_2,l_3)$: We use the
		following function on $\Omega$,
		\begin{equation}
			w(x_1, x_2, x_3) = \cos(2\pi k_1 x_1/l_1) \cos(2\pi k_2 x_2/l_2)
			\cos(\pi k_3 (x_3+l_3^-)/l_3)\,,
		\end{equation}
		and set
		\begin{equation}
			\label{eq:whole-case-fields}
			E(x) = \dvec{0}{\del_3w}{-\del_2 w}\,,\qquad 
			H(x) = (i\omega\mu)^{-1}\,
			\dvec{-\del_2^2 w - \del_3^2 w}{\del_1 \del_2 w}{\del_1 \del_3 w}\,.
		\end{equation}
		We note that $E_1 = 0$ on $\Gamma$ is satisfied and $E_1 = E_2 = 0$
		holds along horizontal boundaries. Furthermore, $H_1$ has no jump
		across $\Gamma$.
		
		\smallskip Case
		$\omega^2 \in\sigma_M \setminus \sigma(l_1,l_2,l_3)$: In this case,
		for either the symbol ``$+$'' or the symbol ``$-$'' (fixed from now
		on), there holds
		$\omega^2 \in \sigma(l_1,l_2,l_3^\pm)\setminus
		\sigma(l_1,l_2,l_3)$. We note that $k_3 \neq 0$ since otherwise
		$\omega^2 \in \sigma(l_1,l_2,l_3)$. We use
		\begin{equation}
			\label{eq:w-ex-3674}
			w(x_1, x_2, x_3) = \cos(2\pi k_1 x_1/l_1) \cos(2\pi k_2 x_2/l_2)
			\sin(\pi k_3 x_3/l_3^\pm)\,,
		\end{equation}
		which is not vanishing identically because of $k_3 \neq 0$, and set
		\begin{equation}
			\label{eq:half-case-fields}
			H(x) = \one{\Omega_\pm}\dvec{0}{\del_3w}{-\del_2 w}\,,\qquad 
			E(x) = \one{\Omega_\pm}(-i\omega\eps)^{-1}\,
			\dvec{-\del_2^2 w - \del_3^2 w}{\del_1 \del_2 w}{\del_1 \del_3 w}\,.
		\end{equation}
		Note that $E_1 = 0$ on $\Gamma$ is satisfied and $E_1 = E_2 = 0$
		holds along the horizontal boundary $\{x_3 = \pm l_3^\pm\}$ by
		$\sin(\pi k_3 (\pm l_3^\pm)/l_3^\pm) = 0$. Furthermore, $H_1$ has
		no jump across $\Gamma$.
	\end{lemma}
	
	\begin{remark}[On the spectrum of the polarization interface problem]
		The spectrum of the Maxwell system \eqref{eq:Maxwell:strong} without
		interface (obtained by demanding
		$\llbracket H_2 \rrbracket_\Gamma = 0$ instead of
		\eqref{eq:Maxwell:strong:3}), is given by $\sigma(l_1,l_2,l_3)$.
		The spectrum of the Maxwell system \eqref{eq:Maxwell:strong} with a
		full reflection interface condition (obtained by demanding
		$E_2|_\Gamma = 0$ instead of \eqref{eq:Maxwell:strong:4}), is
		$\sigma(l_1,l_2,l_3^+) \cup \sigma(l_1,l_2,l_3^-)$. Thus, for this
		simple geometry, the spectrum for the polarization interface
		condition is the union of these spectra.
		
		We include the following warning: It is not true that every
		eigenfunction for the full reflection condition is also an
		eigenfunctions for the polarization condition. An example is constructed with
		$w$ from \eqref {eq:w-ex-3674} (choosing either $+$ or $-$) by setting
		\begin{equation}
		\label{eq:half-case-fields-full-refl}
		H(x) = \one{\Omega_\pm}\dvec{-\del_3w}{0}{\del_1 w}\,,\qquad 
		E(x) = \one{\Omega_\pm}(-i\omega\eps)^{-1}\,
		\dvec{\del_1 \del_2 w}{-\del_1^2 w - \del_3^2 w}{\del_2 \del_3 w}\,.
		\end{equation}
		This pair satisfies the Maxwell equations in $\Omega\setminus\Gamma$, 
		$E_1$ and $E_2$ vanish along $\Gamma$, it is therefore a solution of the full reflection problem. 
		On the other hand, $H_1$ has a jump across $\Gamma$, the above fields are 
		therefore no solution to the polarization problem.
	\end{remark}
	
	The remainder of this article is structured as follows: We provide
	solution concepts for the Maxwell system in Section~\ref
	{sec.Maxwell}, discuss two different solution concepts for the
	Helmholtz-type system in Section~\ref {sec.Helmholtz}, show the
	equivalence of the two systems in Section~\ref {sec.equivalence}, and
	provide existence and uniqueness results for the Helmholtz system in
	Section~\ref {sec.ex-un-Helmholtz}.

	\section{Maxwell system}
	\label{sec.Maxwell}
	
	In this section, we introduce and analyze the weak solution concept
	for the Maxwell system \eqref{eq:Maxwell:strong}. We consider always
	the geometry of \eqref{eq:def:I1-I2-I3}--\eqref{eq:def:Omega}.

	\subsection{Weak form of the Maxwell system}
	We define the periodic extension of the sets $\Omega$ and $\Gamma$ in
	the $(e_1, e_2)$-plane by
	\begin{align*}
		\Omega_\# \coloneqq \R^2 \times \{-l_3^-, l_3^+\} \,,
		\qquad \Gamma_\# \coloneqq \R^2 \times \{0\} \,.
	\end{align*}
	Every function $u$ on $\Omega$ is identified with its periodic
	extension $\tilde{u}$ on $\Omega_\#$, defined with the rule
	\begin{align*}
		\tilde{u}(k_1 l_1 + k_2 l_2 + x)=u(x) \qquad
		\forall k_1 ,k_2 \in \Z \,,\ x \in \Omega\,.
	\end{align*}
	In order to impose the periodicity condition
	\eqref{eq:Maxwell:strong:6}, we introduce the function spaces
	\begin{align*}
		H_\#(\curl, \Omega)
		&\coloneqq \{ u \in H(\curl, \Omega) \mid \tilde{u} \in H_\loc(\curl, \Omega_\#)\}\,,
		\\
		H_\#(\curl, \Omega\setminus \Gamma)
		&\coloneqq \{ u \in H(\curl, \Omega\setminus \Gamma) \mid \tilde{u}
		\in H_\loc(\curl, \Omega_\#\setminus \Gamma_\#)\}\,.
	\end{align*}
	These two spaces impose the periodicity conditions in $x_1$- and
	$x_2$-direction. The second space is larger than the first space,
	functions in $H_\#(\curl, \Omega\setminus \Gamma)$ can have a jump
	across $\Gamma$ (in all components).
	
	Using trace theorems, one can conclude that the
	traces of tangential components are periodic, e.g.:
	$E_3(x_1 = l_1, x_2, x_3) = E_3(x_1 = 0, x_2, x_3)$ for almost every
	$x_2, x_3$ in the sense of traces. As a warning, we note that the two
	properties $u \in H_\loc(\curl, \Omega_\#)$ and periodicity together
	do not imply $u|_\Omega \in H(\curl, \Omega)$; the reason is that the
	$H(\curl, \Omega)$-norm could still be unbounded. Thus, our
	requirement $u \in H(\curl, \Omega)$ is an additional assumption.
	
	In order to incorporate the interface condition
	\eqref{eq:Maxwell:strong:3} and the boundary conditions
	\eqref{eq:Maxwell:strong:5} for the electric field $E$, we define a
	subspace $X \subset H_\#(\curl, \Omega)$. To introduce the interface
	condition \eqref{eq:Maxwell:strong:4} for the magnetic field, we
	define a subspace $Y \in H_\#(\curl, \Omega \setminus \Gamma)$. In
	order to define the spaces $X$ and $Y$, we use spaces of smooth
	functions. In the subsequent table, the second column indicates
	whether or not it is demanded that functions vanish in a neighborhood
	of the horizontal boundaries
	$\delh = \partial_\textrm{top}\Omega \cup
	\partial_\textrm{bot}\Omega$. The last column indicates whether or not
	the functions can have a jump across $\Gamma$ (``jump'' or ``no j.'');
	furthermore, it can be demanded that the function vanishes in a
	neighborhood of $\Gamma$ (``=0'').
	
	\begin{subequations}
		\renewcommand{\arraystretch}{1.4} 
		\begin{align*}
			\begin{array}{l|c|c}
				\text{Space}
				& \delh
				& \Gamma
				\\
				\hline	
				D_\#(\overline{\Omega})
				\coloneqq \left\{u \in C^\infty\left(\overline{\Omega}\right) \,\middle|\, \tilde{u}
				\in C^\infty\left(\overline{\Omega_\#}\right) \right\}
				& \neq 0 &\textrm{no j.}
				\\
				D_\#(\Omega)
				\coloneqq\left\{ u \in C^\infty(\overline{\Omega}) \,\middle|\,
				\supp(\tilde u)\cap \bar\Omega \subset \Omega_\# \text{ is compact} \right\}
				& = 0 &\textrm{no j.}
				\\
				D_\#(\Omega \setminus \Gamma)
				\coloneqq\left\{ u \in C^\infty(\overline{\Omega}) \,\middle|\,
				\supp(\tilde u)\cap \bar\Omega\subset \Omega_\# \setminus \Gamma_\#
				\text{ is compact} \right\}
				& = 0 &=0
				\\
				D_\#(\overline{\Omega}; \Gamma)
				\coloneqq \left\{u \colon \overline{\Omega} \to \C \,\middle|\, 
				\tilde{u}|_{\R^2\times I_3^\pm}
				\in C^\infty\left(\R^2 \times \overline{I_3^\pm} \right) \right\} & \neq 0
				&\textrm{jump}
			\end{array}
		\end{align*}
	\end{subequations}
	
	With these spaces of smooth functions, we can now define the solution
	spaces $X$ and $Y$. Essentially, the space $Y$ contains functions
	$u \in H_\#(\curl, \Omega\setminus \Gamma)$ such that $u_1$ does not
	jump across $\Gamma$; the component $u_2$ might jump across
	$\Gamma$. The space $X$ contains functions $u \in H_\#(\curl, \Omega)$
	such that $u_1$ and $u_2$ vanish at $\delh$ and the component $u_1$
	vanishes also along $\Gamma$.
	
	\begin{subequations}\label{eq:Def:XY}
		\begin{align}
			\label{eq:Def:Y}
			Y &\coloneqq \bigg\{ u \in H_\#(\curl, \Omega\setminus \Gamma)
			\,\bigg|\, \int_{\Omega\setminus \Gamma} \curl u \cdot \phi
			= \int_{\Omega\setminus \Gamma} u \cdot \curl \phi \, 
			\\
			&\hspace{1.5cm}\forall\, \phi =(\phi_1, \phi_2, \phi_3) \, ,\, \phi_1 \in
			D_\#(\Omega\setminus\Gamma)\,,\phi_2 \in D_\#(\Omega)\,,
			\phi_3 \in D_\#(\Omega)\bigg\}\,,\nonumber
			\\
			\label{eq:Def:X}
			X &\coloneqq \left\{ u \in H_\#(\curl, \Omega) \,\middle|\,
			\int_{\Omega\setminus \Gamma} \curl u \cdot \psi = \int_{\Omega\setminus \Gamma}
			u \cdot \curl \psi \ \forall \,\psi \in Y \right\}\,.
		\end{align}
	\end{subequations}
	
	We formulated the definition in such a way that $X$ is characterized
	with the help of $Y$. Indeed, one can also characterize $Y$ in terms
	of $X$. This illustrates the duality of the two spaces.

	\begin{lemma}[Characterization of $Y$ in terms of $X$]
		The two spaces $X$ and $Y$ of \eqref{eq:Def:XY} satisfy
		\begin{equation}\label{eq:Deff:Y}
			Y = \left\{ u \in H_\#(\curl, \Omega\setminus \Gamma)
			\middle| \int_{\Omega\setminus \Gamma} \curl u \cdot \phi
			= \int_{\Omega\setminus \Gamma} u \cdot \curl \phi \ \forall\, \phi \in X\right\}\,.
		\end{equation}
	\end{lemma}
	
	The proof of the lemma is given in the appendix. It is based on a
	density argument.
	
	\begin{remark}
		Tangential traces of functions $E\in X$ and $H\in Y$ are well
		defined, see Theorem 3.29 in \cite{Monk2003a}. In particular, in the
		sense of traces, we may write $E\times n = 0$ on $\delh$ and
		$E_1 = 0$ on $\Gamma$. Similarly, $H$ satisfies
		$\llbracket H_1\rrbracket_\Gamma = 0$.
	\end{remark} 
	
	With the help of $X$ and $Y$ we can now formulate the weak solution
	concept for the Maxwell system \eqref{eq:Maxwell:strong}. The
	motivation of our definition is: We use test-functions $\psi \in Y$ in
	\eqref{eq:Maxwell:strong:1} and test-functions $\phi \in X$ in
	\eqref{eq:Maxwell:strong:2}.
	
	\begin{definition}[Weak solutions of the Maxwell
		system]\label{def:Maxwell:weak}
		Let $(f_h, f_e) \in L^2(\Omega, \C^3)^2$. We say
		$(E, H) \in L^2(\Omega, \C^3) \times L^2(\Omega, \C^3)$ is a
		solution to the Maxwell system \eqref{eq:Maxwell:strong} if
		\begin{subequations}\label{eq:Maxwell:weak}
			\begin{align}\label{eq:Maxwell:weak:E}
				\int_{\Omega\setminus \Gamma} E \cdot \curl \psi
				&= \int_{\Omega} (i \omega \mu H + f_h) \cdot \psi &&\forall \psi \in Y \,,
				\\\label{eq:Maxwell:weak:H}
				\int_\Omega H \cdot \curl \phi
				&= \int_\Omega (-i \omega \eps E + f_e) \cdot \phi&&\forall \phi \in X \,.
			\end{align}
		\end{subequations}
	\end{definition}
	
	We obtain in the next subsection that every solution
	$(E,H) \in L^2(\Omega, \C^3)^2$ of \eqref{eq:Maxwell:weak} satisfies
	indeed $E\in X$ and $H\in Y$, see
	Lemma~\ref{lem:Regularity-Maxwell}. This shows: We could have demanded
	$E\in X$ and $H\in Y$ in the above definition without changing the
	solution concept.

	\subsection{Regularity of solutions to the Maxwell system}
	In this section, it is relevant to consider also divergence-free right-hand sides. We use the following construction. The spaces $X_0$ and
	$Y_0$ of gradients are given as
	\begin{equation}\label{eq:def:X0-Y0}
		X_0 \coloneqq \{\nabla \varphi \in X \mid \varphi \in H^1_{0,\#}(\Omega) \}\,,
		\qquad
		Y_0 \coloneqq \{\nabla \varphi \in Y \mid \varphi \in H^1_\#(\Omega\setminus \Gamma) \} 
		\,,
	\end{equation}
	where $H^1_{0,\#}(\Omega)$ and $H^1_\#(\Omega\setminus \Gamma)$ are
	the spaces of scalar-valued periodic functions, the first with
	homogeneous Dirichlet condition on the horizontal boundaries, the
	second contains functions that jump across $\Gamma$, for precise
	definitions see Section~\ref {sec.Helmholtz}. The index $0$ of $X_0$ and $Y_0$ indicates, that the are the subsets of $X$ and $Y$, respectively, with vanishing curl. The
	$L^2(\Omega)$-orthogonal complements are $X_0^\perp$ and $Y_0^\perp$.
	In the subsequent result, we make a statement on divergence-free
	right-hand sides; more precisely, we will demand $f_h \in Y_0^\perp$
	and $f_e\in X_0^\perp$. This ensures that $E \in X_0^\perp$ and
	$H\in Y_0^\perp$, i.e.~
	\begin{align}\label{eq:divE}
		\int_\Omega E \cdot g = 0 \qquad \forall g \in X_0\,,
		\\\label{eq:divH}
		\int_\Omega H \cdot g = 0 \qquad \forall g \in Y_0\,.
	\end{align}

	\begin{lemma}[Regularity of solutions to the Maxwell system]
		\label{lem:Regularity-Maxwell}
		Let $f_h, f_e \in L^2(\Omega, \C^3)$ define a right-hand side. Let
		$(E, H) \in L^2(\Omega, \C^3) \times L^2(\Omega, \C^3)$ be a
		solution to the Maxwell system in the sense of
		Definition~\ref{def:Maxwell:weak}. Then, the solution has the
		property $(E,H) \in X \times Y$. Moreover, if the right-hand side
		satisfies $(f_h, f_e) \in Y_0^\perp \times X_0^\perp$, there
		additionally holds $(E,H) \in X_0^\perp \times Y_0^\perp$.
	\end{lemma}
	
	\begin{proof} By the definition of weak solutions, $(E, H)$ satisfies
		\eqref{eq:Maxwell:weak}. Choosing the test-functions
		$\psi \in D_\#(\Omega, \C^3)$ and
		$\phi \in D_\#(\Omega\setminus \Gamma)$ in \eqref{eq:Maxwell:weak},
		we see that \eqref{eq:Maxwell:strong:1} and
		\eqref{eq:Maxwell:strong:2} are satisfied in the sense of
		distributions. We therefore know the distributional curl of $E$ and
		$H$ and know that they are given by $L^2(\Omega)$-functions, namely,
		$\curl E =i \omega \mu H +f_h \in L^2(\Omega, \C^3)$ and
		$\curl H= -i \omega \eps E +f_e \in L^2(\Omega, \C^3)$. Thus,
		$E \in H_\#(\curl, \Omega )$ and
		$H \in H_\#(\curl, \Omega\setminus \Gamma)$.
		
		The fact that $\curl E$ and $\curl H$ are $L^2$-functions allows us to
		insert \eqref{eq:Maxwell:strong} into \eqref{eq:Maxwell:weak}, which
		yields
		\begin{align*}
			\int_{\Omega\setminus \Gamma}
			E \cdot \curl \psi
			&= \int_{\Omega\setminus \Gamma} \curl E \cdot \psi \qquad \forall \,\psi \in Y \,,
			\\
			\int_{\Omega\setminus \Gamma} H \cdot \curl \phi
			&= \int_{\Omega\setminus \Gamma} \curl H \cdot \phi \qquad \forall \,\phi \in X \,.
		\end{align*}
		The first equation verifies $E \in X$, see \eqref{eq:Def:X}. The
		second equation implies $H \in Y$, since the definition of $Y$ in
		\eqref{eq:Def:Y} considers test-functions $\phi$ only in a subset of
		$X$.
		
		\smallskip We consider now
		$(f_h, f_e) \in Y_0^\perp \times X_0^\perp$. In order to show
		$(E,H) \in X_0^\perp \times Y_0^\perp$, we choose test-functions of
		the form $\phi = \nabla \varphi \in X_0$ and
		$\psi = \nabla \eta \in Y_0$ in \eqref{eq:Maxwell:weak}. Since the
		curl of gradients vanishes, the left-hand sides of
		\eqref{eq:Maxwell:weak} vanish. Moreover, due to the orthogonality
		of $Y_0^\perp \ni f_h \perp \nabla \varphi \in Y_0$ and
		$X_0^\perp \ni f_e \perp \nabla \eta \in X_0$ the last terms on the
		right-hand sides vanish and we obtain
		\begin{align*}
			0=\int_\Omega i \omega \mu H \cdot \psi \,\ \forall \psi \in Y_0 \,,
			\qquad 0=\int_\Omega -i \omega \eps E \cdot \phi \,\ \forall \phi \in X_0\,,
		\end{align*} which shows that $(\eps E, \mu H) \in X_0^\perp \times Y_0^\perp$ and, thus, $(E, H) \in X_0^\perp \times Y_0^\perp$.
	\end{proof}
	
	We note that, for arbitrary right-hand sides
	$f_h, f_e \in L^2(\Omega, \C^3)$, the Maxwell-system can be always
	replaced by a system with $(f_h, f_e) \in Y_0^\perp \times X_0^\perp$.
	
	\begin{lemma}[Non-divergence-free right-hand sides]
		\label{lem:Helmholtz:Decomposition:Data}
		Let $f_h, f_e \in L^2(\Omega, \C^3)$ define a right-hand side for
		the Maxwell system, $\eps$ and $\mu$ independent of $x$. We consider
		the Helmholtz decompositions
		\begin{align*}
			f_h &= \tilde{f}_h + \nabla \Psi \qquad\text{ with }
			\tilde{f}_h \in Y_0^\perp \text{ and } \nabla \Psi \in Y_0\,,
			\\
			f_e &= \tilde{f}_e + \nabla \Phi \qquad \text{ with }
			\tilde{f}_e \in X_0^\perp\text{ and } \nabla \Phi \in X_0\,.
		\end{align*}
		We consider arbitrary functions $E,H \in L^2(\Omega, \C^3)$ and
		modified functions
		$\tilde{E} = E - (i \omega \eps)^{-1} \nabla \Phi$ and
		$\tilde{H} = H + (i \omega \mu)^{-1} \nabla \Psi$. Then there holds:
		$(E,H)$ is a weak solution to the Maxwell system for the right-hand
		side $(f_h, f_e)$ if and only if $(\tilde{E}, \tilde{H})$ is weak a
		solution to the Maxwell system for the right-hand sides
		$(\tilde{f}_h, \tilde{f}_e)$.
	\end{lemma}
	
	\begin{proof}
		Lemma \ref{lem:Helmholtz:Decomposition:Data} is a direct consequence
		of the fact that the curl of a
		gradient vanishes,
		$\curl \nabla\Phi = \curl \nabla\Psi =0$.
	\end{proof}
	
	\begin{remark}[Non-divergence-free right-hand sides and general
		$\eps$, $\mu$]
		\label{rem:Helmholtz:Decomposition:Data:general}
		Lemma \ref{lem:Helmholtz:Decomposition:Data} deals with the case of
		coefficients $\eps$ and $\mu$ that are independent of $x$. For the
		general case, i.e.~$\mu, \eps \in L^\infty(\Omega, \C^{3\times3})$,
		one has to choose different constructions. The well-known approach
		in this general case is to define scalar products by
		$\langle u, v \rangle_\mu \coloneqq \int_\Omega \mu u \cdot
		\overline{v}$ and
		$\langle u, v\rangle_\eps =\int_\Omega \eps u \cdot \overline{v}$,
		and to consider with $Y_0^{\perp_\mu}$ and $X_0^{\perp_\eps}$ the
		orthogonal complements of $Y_0$ and $X_0$ with respect to these
		scalar products. Decompositions of the form
		$(i \omega \mu)^{-1} f_h = (i \omega \mu)^{-1} \tilde{f}_h + \nabla
		\Psi$ with the choice $\tilde{H} = H + \nabla \Psi$ show that it is
		sufficient to consider $f_h\in Y_0^{\perp_\mu}$ and decompositions of the form
		$(i \omega \eps)^{-1} f_e = (i \omega \eps)^{-1} \tilde{f}_e + \nabla
		\Phi$ with the choice $\tilde{E} = E - \nabla \Phi$ show that it is
		sufficient to consider $f_e\in X_0^{\perp_\eps}$.
	\end{remark}

	\section{Helmholtz-type system}
	\label{sec.Helmholtz}
	
	In this section, we analyze the Helmholtz-type system
	\eqref{eq:Helmholtz:strong}. We introduce two different solution
	concepts, weak solutions and very weak solutions. Actually, we will
	see that the two concepts coincide if we restrict the solution space for the very weak solutions to $X \times Y$. Nevertheless, the distinction is useful in
	the result regarding equivalence of Maxwell- and Helmholtz-system, see
	Section~\ref {sec.equivalence}. The solvability properties of the
	Helmholtz-system are easy to obtain, see Section~\ref
	{sec.ex-un-Helmholtz}.

	\subsection{Weak form of the Helmholtz-type system}
	In order to define the weak solution concept for
	\eqref{eq:Helmholtz:strong}, we use spaces of periodic functions. As
	before, for an arbitrary function $u \colon \Omega\to \C$, we write
	$\tilde u$ for the periodic extension of $u$. We define, in this
	order: periodic functions (no restrictions regarding interface or
	horizontal boundaries), periodic functions that vanish along $\delh$,
	periodic functions that can jump across $\Gamma$, periodic functions
	that vanish along $\delh$ and along $\Gamma$:
	\begin{align*}
		H^1_{\#}(\Omega) &\coloneqq \{u \in H^1(\Omega) \mid \tilde{u} \in H^1_\loc(\Omega_\#) \}\,,
		\\
		H^1_{0,\#}(\Omega) &\coloneqq \left\{u \in H^1_{\#}(\Omega) \middle| \int_\Omega \nabla u \cdot \psi = -\int_\Omega u \div(\psi) \ \forall \, \psi \in H^1_{\#}(\Omega,\C^3) \right\}\,,
		\\
		H^1_{\#}(\Omega\setminus \Gamma) &\coloneqq \{u \in H^1(\Omega\setminus \Gamma) \mid \tilde{u} \in H^1_\loc(\Omega_\# \setminus \Gamma_\#)\}\,,
		\\
		H^1_{0,\#}(\Omega\setminus \Gamma) &\coloneqq \left\{u \in H^1_{\#}(\Omega) \middle| \int_\Omega \nabla u \cdot \psi = -\int_\Omega u \div(\psi) \ \forall \, \psi \in H^1_{\#}(\Omega\setminus \Gamma, \C^3) \right\}\,.
	\end{align*}
	
	In the Helmholtz-type system \eqref{eq:Helmholtz:strong}, the last
	four equations contain only derivatives in the $x_1$-direction. These
	equations are ordinary differential equations with solutions that
	depend on $x_1$, for every parameter input
	$(x_2,x_3)\in I_2\times I_3$. Accordingly, for the solution concept,
	we need spaces of functions that have some $x_1$-regularity, but not
	necessarily regularity in the other coordinates.
	
	Similar to the above constructions, we associate, to every function
	$u \colon I_1 \to \C$, its periodic extension $\tilde{u} \colon \R\to \C$,
	defined through
	\begin{align*}
		\tilde u(k_1 l_1 + x_1) = u(x_1) \qquad \forall k_1 \in \Z\,,\ x_1\in I_1\,.
	\end{align*}
	We can then define
	\begin{align*}
		&D_\#(I_1) \coloneqq \{u \in C^\infty(\overline{I_1}) \mid \tilde{u} \in C^\infty(\R) \}\,,
		\\
		&H^1_{\#}(I_1) \coloneqq \{u \in H^1(I_1) \mid \tilde{u} \in H^1_\loc(\R) \}\,.
	\end{align*}
	The cross sections of $\Omega$, $\Omega\setminus \Gamma$ and $\Gamma$
	are denoted as
	\begin{align}\label{eq:def:U-Gamma_U}
		U \coloneqq I_2 \times I_3 \,, \qquad
		\Gamma_U \coloneqq I_2 \times \{0\} \,.
	\end{align}
	As solution spaces for the Helmholtz-type system we use
	\begin{subequations}
		\begin{align}
			W &\coloneqq L^2(U, H^1_\#(I_1))
			\\
			X_H &\coloneqq \left\{ E\in L^2(\Omega, \C^3) \,\middle|\,
			E_1\in H^1_{0,\#}(\Omega \setminus \Gamma)\,,\ E_2, E_3\in W \right\}\,,
			\\
			Y_H &\coloneqq \left\{ H\in L^2(\Omega, \C^3) \,\middle|\,
			H_1\in H^1_{\#}(\Omega)\,,\ H_2, H_3\in W \right\}\,.
		\end{align}
	\end{subequations}
	
	We will consider right-hand sides
	$(f_h , f_e) \in L^2(\Omega, \C^3)^2$ in the Helmholtz-type
	system (when, we show the equivalence to the Maxwell system, we consider only $(f_h , f_e) \in Y_0^\perp \times X_0^\perp$). Given such functions, we define linear forms
	$F_{E_1} \in H^1_{0,\#}(\Omega\setminus \Gamma)'$,
	$F_{H_1} \in H^1_\#(\Omega)'$ and
	$F_{E_2}, F_{H_2}, F_{E_3}, F_{H_3} \in W'$ as follows:
	\begin{subequations}\label{eq:def:F}
		\begin{align}
			\langle F_{E_1}, \phi_1 \rangle
			&\coloneqq \int_\Omega i \omega \mu (f_e)_1 \phi_1
			+ f_h \cdot \curl (e_1 \phi_1)
			&& \forall \phi_1 \in H^1_{0,\#}(\Omega\setminus \Gamma)\,,
			\\
			\langle F_{H_1}, \psi_1 \rangle
			&\coloneqq 		
			\int_\Omega -i \omega \eps (f_h)_1 \psi_1 + f_e \cdot \curl (e_1 \psi_1)
			&& \forall \psi_1 \in H^1_\#(\Omega)\,,
			\\
			\langle F_{E_2}, \phi_2 \rangle
			&\coloneqq \int_\Omega(f_h)_3 \partial_1 \phi_2
			+ i \omega \mu (f_e)_2 \phi_2 && \forall \phi_2 \in W\,,
			\\
			\langle F_{H_2}, \psi_2 \rangle
			&\coloneqq \int_\Omega(f_e)_3 \partial_1 \psi_2 - i \omega \eps (f_h)_2 \psi_2
			&& \forall \psi_2 \in W\,,
			\\
			\langle F_{E_3}, \phi_3 \rangle
			&\coloneqq \int_\Omega -(f_h)_2 \partial_1 \phi_3
			+ i \omega \mu (f_e)_3 \phi_3 && \forall \phi_3 \in W\,,
			\\
			\langle F_{H_3}, \psi_3 \rangle
			&\coloneqq \int_\Omega -(f_e)_2 \partial_1 \psi_3
			- i \omega \eps (f_h)_3 \psi_3 && \forall \psi_3 \in W\,.
		\end{align}
	\end{subequations}

	\begin{definition}[Weak solution of the Helmholtz-type
		system]\label{def:Helmholtz:weak}
		We consider a right-hand side $(f_h , f_e) \in L^2(\Omega,
		\C^3)^2$. We say $(E, H) \in X_H \times Y_H$ is a weak solution of
		the Helmholtz-type system \eqref{eq:Helmholtz:strong} if, for all
		$\phi \in X_H$ and all $\psi \in Y_H$:
		\begin{subequations}\label{eq:Helmholtz:weak}
			\begin{align}\label{eq:Helmholtz:weak:E1}
				\int\limits_\Omega \left\{ \nabla E_1 \cdot \nabla \phi_1
				- \omega^2 \eps \mu E_1 \phi_1\right\}
				&= \langle F_{E_1}, \phi_1 \rangle \,,
				\\
				\label{eq:Helmholtz:weak:H1}
				\int\limits_\Omega \left\{\nabla H_1 \cdot \nabla \psi_1
				- \omega^2 \eps \mu H_1 \psi_1\right\} &= \langle F_{H_1}, \psi_1 \rangle \,,
				\\
				\label{eq:Helmholtz:weak:E2}
				\int\limits_{\Omega } \left\{\partial_1 E_2 \partial_1 \phi_2 - \omega^2 \eps \mu E_2 \phi_2\right\}
				&= \int\limits_{\Omega } \left\{\partial_2 E_1 \partial_1 \phi_2 - i \omega \mu \partial_3 H_1 \phi_2\right\} + \langle F_{E_2} , \phi_2 \rangle \,,
				\\
				\label{eq:Helmholtz:weak:H2}
				\int\limits_{\Omega } \left\{\partial_1 H_2 \partial_1 \psi_2 - \omega^2 \eps \mu H_2 \psi_2\right\}
				&= \int\limits_{\Omega } \left\{\partial_2 H_1\partial_1 \psi_2 - i \omega \eps \partial_3 E_1 \psi_2\right\} + \langle F_{H_2} , \psi_2 \rangle \,,
				\\
				\label{eq:Helmholtz:weak:E3}
				\int\limits_{\Omega }\left\{ \partial_1 E_3 \partial_1 \phi_3 - \omega^2 \eps \mu E_3 \phi_3\right\} 
				&= 
				\int\limits_{\Omega } \left\{\partial_3 E_1 \partial_1 \phi_3 + i \omega \mu \partial_2 H_1 \phi_3\right\} +\langle F_{E_3} , \phi_3 \rangle \,,
				\\
				\label{eq:Helmholtz:weak:H3}
				\int\limits_{\Omega }\left\{ \partial_1 H_3 \partial_1 \psi_3 -	\omega^2 \eps \mu H_3 \psi_3 \right\}
				&= 
				\int\limits_{\Omega} \left\{\partial_3 H_1 \partial_1 \psi_3 - i \omega \eps \partial_2 E_1 \psi_3\right\} +\langle F_{H_3} , \psi_3 \rangle \,.
			\end{align}
		\end{subequations}
	\end{definition}
	
	\subsection{Very weak solutions of the Helmholtz-type system}
	
	Additionally to the weak solution concept, we introduce the concept of
	very weak solutions. Essentially, in the very weak concept, all
	derivatives are moved to the test-functions, it therefore has the
	character of a distributional concept. We do not call it a
	distributional solution since test-functions do not necessarily have a
	compact support. Indeed, we encode some boundary and interface
	conditions with test-functions that are not vanishing at boundaries.

	\begin{definition}[Very weak solution of the Helmholtz-type problem]
		\label{def:Helmholtz:Vweak}
		Let $(f_h , f_e) \in L^2(\Omega, \C^3)^2$ and
		$F_{E_1}, F_{H_1}, F_{E_2}, F_{H_2}, F_{E_3}, F_{H_3} $ be given by
		\eqref{eq:def:F}. We say that
		$(E,H) \in L^2(\Omega, \C^3) \times L^2(\Omega, \C^3)$ is a very
		weak solution to the Helmholtz-type problem
		\eqref{eq:Helmholtz:strong} if
		\begin{subequations}\label{eq:Helmholtz:Vweak}
			\begin{align}\label{eq:Helmholtz:Vweak:E1}
				\int\limits_\Omega \left\{- E_1 \Delta \phi_1 - \omega^2 \eps \mu E_1 \phi_1\right\}
				&= \langle F_{E_1}, \phi_1 \rangle \,,
				\\
				\label{eq:Helmholtz:Vweak:H1}
				\int\limits_\Omega\left\{ - H_1 \Delta \psi_1 - \omega^2 \eps \mu H_1 \psi_1\right\}
				&= \langle F_{H_1}, \psi_1 \rangle \,,
				\\
				\label{eq:Helmholtz:Vweak:E2}
				\int\limits_{\Omega }\left\{ - E_2 \partial_1^2 \phi_2 - \omega^2 \eps \mu E_2 \phi_2\right\}
				&= \int\limits_{\Omega } \left\{-E_1 \partial_2 \partial_1 \phi_2 + i \omega \mu H_1 \partial_3 \phi_2 \right\}+\langle F_{E_2}, \phi_2 \rangle \,,
				\\
				\label{eq:Helmholtz:Vweak:H2}
				\int\limits_{\Omega } \left\{- H_2 \partial_1^2 \psi_2 - \omega^2 \eps \mu H_2 \psi_2\right\}
				&= \int\limits_{\Omega } \left\{-H_1 \partial_2 \partial_1 \psi_2 - i \omega \eps E_1 \partial_3 \psi_2\right\} +\langle F_{H_2}, \psi_2 \rangle \,,
				\\
				\label{eq:Helmholtz:Vweak:E3}
				\int\limits_{\Omega } \left\{-E_3 \partial_1^2 \phi_3 - \omega^2 \eps \mu E_3 \phi_3 \right\}
				&= 
				\int\limits_{\Omega } \left\{-E_1 \partial_3 \partial_1 \phi_3 - i \omega \mu H_1 \partial_2 \phi_3\right\} +\langle F_{E_3}, \phi_3 \rangle \,,
				\\
				\label{eq:Helmholtz:Vweak:H3}
				\int\limits_{\Omega } \left\{-H_3 \partial_1^2 \psi_3 -	\omega^2 \eps \mu H_3 \psi_3 \right\}
				&= \int\limits_{\Omega } \left\{-H_1 \partial_3 \partial_1 \psi_3 + i \omega \eps E_1 \partial_2 \psi_3\right\}+ \langle F_{H_3}, \psi_3 \rangle \,,
			\end{align}
			for all test-functions
			\begin{align}
				&\phi_1 \in D_\#(\Omega\setminus \Gamma)\,,\quad \phi_2 \in D_\#(\Omega) \,,\quad
				\phi_3 \in D_\#(\overline{\Omega}; \Gamma) \,,
				\\
				&\psi_1\in \{\psi_1 \in D_\#(\overline{\Omega}) \mid \partial_3 \psi_1 \in D_\#(\Omega)\}\,,
				\quad \psi_2 \in D_\#(\overline{\Omega}; \Gamma) \,,\quad \psi_3 \in D_\#(\Omega) \,.
			\end{align}
		\end{subequations}
	\end{definition}

	\subsection{Equivalence of the weak and very weak Helmholtz solution
		concept}
	
	In this section, we show the equivalence of the solution concepts.
	Lemma \ref{lem:Vweak-Helmholtz-implies-weakHelmholtz} provides that
	every very weak solution of the Helmholtz-type system (Definition
	\ref{def:Helmholtz:Vweak}), which has the additional regularity
	$(E,H) \in X \times Y$, is also a weak solution of the Helmholtz-type
	system (Definition \ref{def:Helmholtz:weak}). Lemma
	\ref{lem:weak-Helmholtz-implies-VweakHelmholtz} provides that every
	weak solution is also a very weak solution.
	
	For the proof of the following lemma, we use another space of periodic
	functions. In the spirit of previous constructions, every function $u$
	on $I_1 \times I_2$ is identified with its periodic extension
	$\tilde{u}$ on $\R^2$, defined as
	\begin{align*}
		\tilde u(k_1 l_1 + k_2 l_2 + x)=u(x) \qquad
		\forall k_1 ,k_2 \in \Z\,,\ x \in I_1 \times I_2\,,
	\end{align*}
	and the corresponding space is
	\begin{align*}
		D_\#(I_1 \times I_2) \coloneqq \{ u \in C^\infty(\overline{I_1} \times \overline{I_2})
		\mid \tilde u\in C^\infty(\R^2) \}\,.
	\end{align*}
	Moreover, we use the spaces
	\begin{align*}
		&D(I_3\setminus \{0\}) \coloneqq \left\{ u \in C^\infty_c(I_3 \setminus\{0\}) \right\}\,,
		\\
		&D(\overline{I_3}; \{0\}) \coloneqq \left\{ u \colon \overline{I}_3 \to \C \, \middle|\, u|_{I_3^\pm} \in C^\infty\left(\overline{I_3^\pm}\right) \right\}\,.
	\end{align*}

	\begin{lemma}[Very weak implies weak]
		\label{lem:Vweak-Helmholtz-implies-weakHelmholtz}
		We consider
		right-hand sides $(f_h, f_e) \in L^2(\Omega, \C^3)^2$. Let
		$(E, H)\in L^2(\Omega, \C^3)^2$ be a very weak solution to the
		Helmholtz-type system of Definition~\ref{def:Helmholtz:Vweak} with
		the regularity $(E, H)\in X \times Y$. Then, $(E,H)$ is also a weak
		solution to the Helmholtz-type system, see
		Definition~\ref{def:Helmholtz:weak}. In particular,
		$(E,H) \in X_H \times Y_H$.
	\end{lemma}
	
	\begin{proof}
		Let $(E,H)\in X \times Y$ be a very weak solution to the
		Helmholtz-type system, i.e., a solution to
		\eqref{eq:Helmholtz:Vweak}. We have to show the regularities
		$E_1 \in H^1_{0,\#}(\Omega \setminus \Gamma)$,
		$H_1 \in H^1_\#(\Omega)$, $E_2$, $H_2$, $E_3$, $H_3 \in W$ and that
		$(E,H)$ satisfies \eqref{eq:Helmholtz:weak}. The systems
		\eqref{eq:Helmholtz:Vweak} and \eqref{eq:Helmholtz:weak} differ
		essentially by some integration by parts and we have to show that
		these integrations by parts are admissible and do not come with
		boundary terms.
		
		\smallskip \textbf{$H^1_{0, \#}(\Omega\setminus \Gamma)$-regularity
			of $E_1$}: Let
		$\zeta \in D_\#(I_1 \times I_2)$ and
		$\varphi \in D(\overline{I_3};\{0\})$ be arbitrary. Noting that
		$E \in X$, we choose $\psi = (0,\zeta \varphi,0) \in Y$ in the
		definition of $X$, see \eqref{eq:Def:X}, which gives
		\begin{align*}
			\int_{\Omega} (\curl E)_2 \zeta \varphi
			= \int_{\Omega} \{-E_1 \zeta \partial_3 \varphi + E_3 \partial_1\zeta \varphi \}
			\,.
		\end{align*}
		We rewrite this with Fubini's theorem as
		\begin{align*}
			\int_{I_3\setminus \{0\}} \left(-\int_{I_1 \times I_2}E_1 \zeta\right) \partial_3 \varphi =
			\int_{I_3\setminus \{0\}} \left(\int_{I_1 \times I_2}
			\{ - E_3 \partial_1\zeta + (\curl E)_2 \zeta\} \right) \varphi \,. 
		\end{align*}
		This relation allows to conclude that the expression
		$\int_{I_1 \times I_2}E_1 \zeta$ has a distributional derivative in
		direction $x_3$ and that the expression is vanishing at the points
		$x_3 = 0$ and $x_3 = \pm l_3^\pm$. Since $\zeta$ was arbitrary, this fact
		encodes that $E_1$ vanishes at the top boundary and along $\Gamma$.
		
		This shows that $E_1$ solves the Poisson-problem
 \eqref{eq:Helmholtz:Vweak:E1} with periodicity and
 Dirichlet boundary conditions. A Weyl-type lemma (see
 Lemma \ref{lem:regularity:Distrib-Boundary-Dirichlet})
 implies the regularity and the boundary conditions
 that are encoded with
 $E_1 \in H^1_{0,\#}(\Omega\setminus \Gamma)$. We note
 that we include Lemma
 \ref{lem:regularity:Distrib-Boundary-Dirichlet} to
 have this text self-contained; it might also be
 possible to conclude with results of Chapter 3 of
 \cite{Kozlov-Mazya-Rossmann-1997}.

		An integration by parts in \eqref{eq:Helmholtz:Vweak:E1} is
		permitted and we obtain \eqref{eq:Helmholtz:weak:E1}.
		
		\smallskip \textbf{$H^1_\#(\Omega)$-regularity of $H_1$}: Let
		$\zeta \in D_\#(I_1 \times I_2)$ and
		$\varphi \in C_c^\infty(I_3)$. Noting that $H \in Y$, we choose
		$\phi = (0,\zeta \varphi,0)$ in the definition of $Y$, see
		\eqref{eq:Def:Y}, and find
		\begin{align*}
			\int_{\Omega} (\curl H)_2 \zeta \varphi = \int_{\Omega}
			\{-H_1 \zeta \partial_3 \varphi + H_3 \partial_1\zeta \varphi \}
			\,.
		\end{align*}
		Since $\varphi \in C_c^\infty(I_3)$ was arbitrary, this equation
		implies
		\begin{equation*}
			\partial_3 \left(\int_{I_1 \times I_2} H_1 \zeta\right) 
			= \int_{I_1 \times I_2}
			\{- H_3 \partial_1\zeta + (\curl H)_2 \zeta \}
			\quad \text{ in } L^2(I_3) \,,
		\end{equation*}
		and, in particular, that the expression
		$\int_{I_1 \times I_2}H_1 \zeta$ has no jump at $x_3 = 0$. Since
		$\zeta \in D_\#(I_1 \times I_2)$ was
		arbitrary, we conclude that $H_1$ has no jump across $\Gamma$.
		Together with the fact that $H_1$ solves
		\eqref{eq:Helmholtz:Vweak:H1}, we deduce $H_1 \in H^1_\#(\Omega)$,
		see Lemma \ref{lem:regularity:Distrib-Boundary-Neumann}. Integration
		by parts is allowed and equation \eqref{eq:Helmholtz:Vweak:H1}
		implies \eqref{eq:Helmholtz:weak:H1}.
		
		\smallskip \textbf{$W$-regularity of $E_2, H_2, E_3 , H_3$}: We
		present the argument for $E_2$. We choose $\phi_2$ of the form
		$\phi_2(x_1,x_2,x_3) = \varphi(x_1)\zeta(x_2,x_3)$ for arbitrary
		$\zeta \in C^\infty_c(U\setminus \Gamma_U)$ (we recall $U = I_2 \times I_3$ and $\Gamma_U = I_2 \times \{0\}$) and
		$\varphi \in D_\#(I_1)$. Due to the compact support of
		$\zeta$ and the $H^1(\Omega\setminus \Gamma)$ regularity of $E_1$
		and $H_1$, we can integrate by parts all terms on the right-hand
		sides of \eqref{eq:Helmholtz:Vweak:E2} that contain $E_1$ and $H_1$
		and derivatives in the direction $e_2$ or $e_3$. We get
		\begin{equation}\label{eq:left-veryWeakHh-right-weakHh}
			\int_{\Omega } \{- E_2 \partial_1^2 \phi_2 - \omega^2 \eps \mu E_2 \phi_2\}
			= \int_{\Omega } \{\partial_2E_1 \partial_1 \phi_2 - i \omega \mu \partial_3H_1 \phi_2 \}
			+ \langle F_{E_2} , \phi_2 \rangle \,.
		\end{equation}
		Since $\zeta \in C^\infty_c(U\setminus \Gamma)$ was arbitrary, and since the functions in the subsequent formula are all of class $L^2(\Omega)$, we obtain the desired
		equations pointwise: For almost every $(x_2,x_3) \in U$, there holds
		\begin{align*}
			\int_{I_1 }
			& \{- E_2(\cdot, x_2,x_3) \partial_1^2 \varphi
			- \omega^2 \eps \mu E_2(\cdot,x_2,x_3) \varphi\}
			\\
			= 
			& \int_{I_1 } \{\partial_2E_1(\cdot,x_2,x_3) \partial_1 \varphi
			- i \omega \mu \partial_3 H_1(\cdot,x_2,x_3) \varphi \}
			\\
			& + 
			\int_{I_1}\left\{(f_h)_3(\cdot,x_2,x_3) \partial_1 \varphi
			+ i \omega \mu (f_e)_2(\cdot,x_2,x_3) \varphi\right\}\,.
		\end{align*}
		The Weyl-Lemma implies, for
		a.e.~$(x_2,x_3) \in U$, there holds
		$E_2(\cdot, x_2,x_3) \in H^1_\#(I_1)$.	
		This can be also obtain from the fact that 	$E_2(\cdot, x_2,x_3)$ solves an ordinary differential equation. Thus, there exists $C>0$ (independent of $x_1$ and $x_2$), such that
			\begin{align*}
				&\| E_2( \cdot ,x_2, x_3) \|_{H^1(I_1) } \leq C \left( \|(f_h)_3(\cdot, x_2, x_3) \|_{L^2(I_1)} + \|(f_e)_2(\cdot, x_2, x_3) \|_{L^2(I_1)} \right. \\
				&\qquad\qquad\qquad \left.+ \| \partial_2 E_1(\cdot, x_2, x_3) \|_{L^2(I_1)} + \|\partial_3 H_1(\cdot, x_2, x_3) \|_{L^2(I_1)}\right)\,.
			\end{align*} 
		Integrating this inequality (or, more precisely, the squared inequality) 
		over $I_2 \times I_3^\pm$ yields, by $L^2(U)$-boundedness
		of the right-hand side:
			\begin{align*}
				\int_{I_2 \times I_3^\pm } \| E_2( \cdot ,x_2, x_3) \|^2_{H^1(I_1)} 
				\, dx_2\, dx_3 \leq C &\left( \|(f_h)_3\|^2_{L^2(\Omega^\pm)}+ \|(f_e)_2 \|^2_{L^2(\Omega^\pm)} \right. \\
				&\quad
				\left. + \|E_1 \|^2_{H^1(\Omega^\pm)} + \|H_1\|^2_{H^1(\Omega^\pm)} \right)\,,
			\end{align*}
		which implies $E_2 \in L^2(U, H^1_\#(I_1))$. We can therefore integrate the first term of the left-hand side of
		\eqref{eq:left-veryWeakHh-right-weakHh} by parts and obtain
		\eqref{eq:Helmholtz:weak:E2} for $\phi_2$ of the form
		$\phi_2(x_1,x_2,x_3) = \varphi(x_1)\zeta(x_2,x_3)$ for
		$\zeta \in C^\infty_c(U\setminus \Gamma_U)$,
		$\varphi \in D_\#(I_1)$. By the density of the span of
		such functions in $W$, \eqref{eq:Helmholtz:weak:E2} holds for arbitrary
		$\phi_2 \in W$.
		
		For the equations of $H_2, E_3 , H_3$, we can use the same argument.
		Thus, $(E,H) \in X_H \times Y_H$ is a weak solution of the
		Helmholtz-type problem \eqref{eq:Helmholtz:weak}.
	\end{proof}

	\begin{lemma}[Weak implies very weak]
		\label{lem:weak-Helmholtz-implies-VweakHelmholtz}
		Let $(f_h, f_e) \in L^2(\Omega, \C^3)^2$ and
		$(E, H)\in X_H \times Y_H$ be a weak solution to the Helmholtz-type
		system, i.e., a solution to \eqref{eq:Helmholtz:weak}. Then, $(E,H)$
		is a very weak solution to the Helmholtz-type system, i.e., a
		solution to \eqref{eq:Helmholtz:Vweak}.
	\end{lemma}
	
	\begin{proof}
		The solutions of \eqref{eq:Helmholtz:Vweak} require less regularity
		than the solutions of \eqref{eq:Helmholtz:weak} and the equations
		itself differ only by some integration by parts. Thus, let
		$(E, H)\in X_H \times Y_H$ be a weak solution to the Helmholtz-type
		system, i.e., a solution to \eqref{eq:Helmholtz:weak}. We have to
		show that all integrations by parts can be performed without
		boundary terms.
		
		The property $E_1 \in H^1_{0,\#}(\Omega \setminus \Gamma)$ allows to
		integrate by parts without boundary terms; we exploit the
		periodicity in $e_1$ and $e_2$ direction and the vanishing trace of
		$E_1$ at $\Gamma$ and $\delh$.
		
		In \eqref{eq:Helmholtz:weak:H1}, the integration by parts for the
		left-hand side is not producing boundary terms since
		$\partial_3 \psi_1 \in D_\#(\Omega)$ and, thus, vanishes at $\delh$
		and is periodic in the $e_1$ and $e_2$ directions.
		For the left-hand sides of
		\eqref{eq:Helmholtz:weak:E2}--\eqref{eq:Helmholtz:weak:H3}, there do
		not arise boundary terms due to the periodicity of the functions in
		the $e_1$-direction.
		
		For the right-hand sides of \eqref{eq:Helmholtz:Vweak:E2} and
		\eqref{eq:Helmholtz:Vweak:H3}, there is no boundary term for $H_1$
		because of $\phi_2 \in D_\#(\Omega)$ and $\psi_3 \in D_\#(\Omega)$.
		For the right-hand sides of \eqref{eq:Helmholtz:Vweak:H2} and
		\eqref{eq:Helmholtz:Vweak:E3}, the boundary term for the $H_1$-term
		vanishes due to the periodicity in the $e_2$-direction of the
		test-functions and $H_1$.
	\end{proof}

	\section{Equivalence of the Maxwell and the
		Helmholtz-type system}
	\label{sec.equivalence}
	
	\begin{lemma}[Maxwell implies Helmholtz]
		\label{lem:Maxwell-implies-Helmholtz}
		We consider right-hand sides
		$(f_h, f_e) \in Y_0^\perp \times X_0^\perp$. Let
		$(E, H)\in X \times Y$ be a solution to the Maxwell system, see
		Definition~\ref{def:Maxwell:weak}. Then, $(E,H)$ is a weak solution
		of the Helmholtz-type system, see
		Definition~\ref{def:Helmholtz:weak}. In particular,
		$(E, H) \in X_H\times Y_H$.
	\end{lemma}

	\begin{proof}
		Let $(E, H) \in L^2(\Omega, \C^3) \times L^2(\Omega, \C^3)$ be a
		solution to the Maxwell system according to
		Definition~\ref{def:Maxwell:weak}, i.e., system
		\eqref{eq:Maxwell:weak} is satisfied. Our aim is to show that
		$(E,H)$ is a very weak solution to the Helmholtz-type system, i.e.,
		a solution to \eqref{eq:Helmholtz:Vweak}. Lemma~\ref
		{lem:Vweak-Helmholtz-implies-weakHelmholtz} then yields that $(E,H)$
		is also a weak solution to the Helmholtz-type system.
		
		\smallskip \textbf{Equation for $E_1$}: Let
		$\varphi_1 \in D(\Omega\setminus \Gamma)$ be arbitrary. We choose
		$\phi = i \omega \mu (\varphi_1,0,0)$ in \eqref{eq:Maxwell:weak:H},
		noting that the condition $\phi \in X$ is satisfied since $\phi$ is
		smooth and its first and second component vanish at $\delh$ and its
		first component vanishes on $\Gamma$. We obtain
		\begin{align*}
			\int_\Omega i \omega \mu H \cdot \curl (\varphi_1,0,0)
			= \int_\Omega (\omega^2 \eps \mu E + i \omega \mu f_e) \cdot (\varphi_1,0,0) \,.
		\end{align*}
		We choose
		$\psi= \curl (\varphi_1,0,0) = (0, \partial_3 \varphi_1, -
		\partial_2 \varphi_1)$ in \eqref{eq:Maxwell:weak:E}, noting that
		$\psi \in Y$ is satisfied since the first component of $\psi$ does
		not jump along $\Gamma$. We find
		\begin{align*}
			\int_\Omega E \cdot \curl \curl (\varphi_1,0,0)
			= \int_\Omega (i \omega \mu H +f_h) \cdot \curl (\varphi_1,0,0) \,.
		\end{align*}
		Summing up the two equations yields
		\begin{equation}\label{eq:Helmholtz:Vweak:E1-curl-curl}
			\int_\Omega \left\{E \cdot \curl \curl (e_1\varphi_1) - \omega^2 \eps \mu E_1 \varphi_1\right\}
			= \int_\Omega \left\{i \omega \mu (f_e)_1 \varphi_1 + f_h \cdot \curl (e_1 \varphi_1)\right\} \,. 
		\end{equation}
		Since we have divergence-free data, Lemma \ref{lem:Regularity-Maxwell} implies that $E \in X_0^\perp$. Therefore, taking $g = \nabla \partial_1 \varphi_1\in X_0$ in \eqref{eq:divE}, we obtain
		\begin{align}\label{eq:divE-tested} 
			\int_\Omega E \cdot \nabla \div (e_1 \varphi_1) = 0\,.
		\end{align}
		Using the identity $\curl \curl = - \Delta + \nabla \div$ and
		subtracting \eqref{eq:divE-tested} from
		\eqref{eq:Helmholtz:Vweak:E1-curl-curl} gives
		\begin{align*}
			\int_\Omega \left\{E \cdot \Delta (e_1\varphi_1) - \omega^2 \eps \mu E_1 \varphi_1\right\}
			= \int_\Omega \left\{i \omega \mu (f_e)_1 \varphi_1 + f_h \cdot \curl (e_1 \varphi_1)\right\} \,.
		\end{align*}
		This provides \eqref{eq:Helmholtz:Vweak:E1} for $\phi_1=\varphi_1$.
		
		\smallskip \textbf{Equation for $H_1$}: The derivation of
		\eqref{eq:Helmholtz:Vweak:H1} is very similar to the derivation of
		\eqref{eq:Helmholtz:Vweak:E1}. Let
		$\eta_1 \in D_\#(\overline{\Omega})$ with
		$\partial_3 \eta_1 \in D_\#(\Omega)$ be arbitrary. We choose
		$\psi =-i \omega \eps (\eta_1,0,0) \in Y$ in
		\eqref{eq:Maxwell:weak:E}, where $\psi \in Y$ since $\eta_1$ does
		not jump across $\Gamma$. We choose
		$\phi =\curl (e_1 \eta_1) = (0,\partial_3 \eta_1, -\partial_2
		\eta_1) \in X$ in \eqref{eq:Maxwell:weak:H}, where $\phi \in X$
		since $\partial_3 \eta_1 \in D_\#(\Omega)$ and, thus, the second
		component of $\phi$ vanishes at $\delh$. Moreover, we choose
		$g = -\nabla \partial_1 \eta_1 \in Y_0$ in \eqref{eq:divH}. Adding
		the resulting three equations and using the identity
		$\curl \curl = - \Delta + \nabla \div$ gives
		\eqref{eq:Helmholtz:Vweak:H1} for $\psi_1 = \eta_1$.
		
		\smallskip \textbf{Equation for $E_2$:} Let
		$\varphi_2 \in D_\#(\Omega)$ be arbitrary. We choose
		$\phi = i \omega \mu (0,\varphi_2,0) \in X$ in
		\eqref{eq:Maxwell:weak:H} and obtain:
		\begin{align*}
			i \omega \mu \int_\Omega \left\{- H_1 \partial_3 \varphi_2
			+ H_3 \partial_1 \varphi_2 \right\}
			= \int_\Omega \left\{\omega^2 \eps \mu E_2 \varphi_2 + i \omega \mu (f_e)_2 \varphi_2\right\}\,.
		\end{align*}
		Choosing $\psi = (0,0, \partial_1 \varphi_2) \in Y$ in
		\eqref{eq:Maxwell:weak:H} gives: 
		\begin{align*}
			\int_\Omega \left\{E_1 \partial_2 \partial_1 \varphi_2
			- E_2 \partial_1\partial_1 \varphi_2\right\}
			= \int_\Omega \left\{i \omega \mu H_3 \partial_1 \varphi_2 + (f_h)_3 \partial_1 \varphi_2\right\}\,.
		\end{align*}
		Summing these equations gives \eqref{eq:Helmholtz:Vweak:E2} for
		$\phi_2 = \varphi_2$.
		
		\smallskip \textbf{Equations for $H_2$, $E_3$, $H_3$:} These
		equations are derived as the equation for $E_2$. We only sketch the
		necessary steps.
		
		\smallskip \textit{Equation for $H_2$:} Let
		$\eta_2 \in D_\#(\overline{\Omega}; \Gamma)$ be arbitrary. We choose
		$\psi= i \omega \eps (0,\eta_2,0) \in Y$ in
		\eqref{eq:Maxwell:weak:E} and $\phi = (0,0,\partial_1 \eta_2) \in X$
		in \eqref{eq:Maxwell:weak:H}. Summing the resulting equations gives
		\eqref{eq:Helmholtz:Vweak:H2} for $\psi_2 = \eta_2$.
		
		\textit{Equation for $E_3$:} Let
		$\varphi_3 \in D_\#(\overline{\Omega}; \Gamma)$ be arbitrary. We
		choose $\phi =i \omega \mu (0,0,\varphi_3) \in X$ in
		\eqref{eq:Maxwell:weak:H} and
		$\psi=(0,\partial_1 \varphi_3,0) \in Y$ in
		\eqref{eq:Maxwell:weak:E}. Summing the resulting equations gives
		\eqref{eq:Helmholtz:Vweak:E3} for $\phi_3 = \varphi_3$.
		
		\textit{Equation for $H_3$:} Let $\eta_3 \in D_\#(\Omega)$ be
		arbitrary. We choose $\psi = i \omega \eps (0,0, \eta_3) \in Y$ in
		\eqref{eq:Maxwell:weak:E} and $\phi = (0,\partial_1 \eta_3,0) \in X$
		in \eqref{eq:Maxwell:weak:H}. Summing the resulting equations gives
		\eqref{eq:Helmholtz:Vweak:H3} for $\psi_3 = \eta_3$.
	\end{proof}
	
	With the above result, we know that Maxwell solutions are also
	Helmholtz solutions. We next want to show the opposite implication for the case $\omega^2 \notin \sigma(l_1)$:
	Weak solution of the Helmholtz-type system are also a solution to the
	Maxwell system.
	
	In this derivation, we use spaces of smooth functions on the cross
	section $U = I_2 \times I_3$ of $\Omega$. These spaces corresponds to
	the subsets of $D_\#(\overline{\Omega})$, $D_\#(\Omega)$,
	$D_\#(\Omega\setminus \Gamma)$ and $D_\#(\overline{\Omega}; \Gamma)$
	of functions that are independent of $x_1$. In the spirit of previous
	constructions, every function $u$ on $U$ is identified with its
	periodic extension $\tilde{u}$ on $\R \times I_3$, defined as
	\begin{align*}
		\tilde u(k_1 l_2 + x)=u(x) \qquad
		\forall k_1\in \Z\,,\ x \in I_2 \times I_3\,,
	\end{align*}
	In the subsequent table, the second column indicates whether or not it
	is demanded that functions vanish in a neighborhood of the horizontal
	boundaries $\delhU$. The last column indicates whether or not the
	functions can have a jump across $\Gamma_U$, or if it is demanded that
	the function vanishes in a neighborhood of $\Gamma_U$.
	
	\begin{subequations}
		\renewcommand{\arraystretch}{1.4} 
		\begin{align*}
			\hspace{-0.1cm}
			\begin{array}{l|c|c}
				\text{Space}
				& \!\delhU\!
				& \!\Gamma_U\!
				\\
				\hline	
				D_\#(\overline{U})
				\coloneqq \left\{u \in C^\infty\left(\overline{U}\right) \,\middle|\, \tilde{u}
				\in C^\infty\left(\overline{U_\#}\right) \right\}
				& \neq 0 &\textrm{no j.}
				\\
				D_\#(U)
				\coloneqq\left\{ u \in C^\infty(\overline{U}) \,\middle|\,
				\supp(\tilde u)\cap \bar U \subset U_\# \text{ is compact} \right\}
				& = 0 &\textrm{no j.}
				\\
				D_\#(U \setminus \Gamma_U)
				\coloneqq\left\{ u \in C^\infty(\overline{U}) \,\middle|\,
				\supp(\tilde u)\cap \bar U\subset U_\# \setminus \Gamma_{U,\#}
				\text{ is compact} \right\} \!
				& = 0 &=0
				\\
				D_\#(\overline{U}; \Gamma_U)
				\coloneqq \left\{u \colon \overline{U} \to \C \,\middle|\, 
				\tilde{u}|_{\R\times I_3^\pm}
				\in C^\infty\left(\R \times \overline{I_3^\pm} \right) \right\} & \neq 0
				&\textrm{jump}
			\end{array}
		\end{align*}
	\end{subequations}

	\begin{lemma}[Helmholtz implies Maxwell]
		\label{lem:Helmholtz-implies-Maxwell}
 We consider right-hand sides
		$(f_h, f_e) \in Y_0^\perp \times X_0^\perp$ and
		$\omega^2 \notin \sigma(l_1)$. Let
		$(E, H)\in X_H \times Y_H$ be a solution to the Helmholtz-type system, see
		Definition~\ref{def:Helmholtz:weak}. Then, $(E,H)$ is a weak solution
		of the Maxwell system, see
		Definition~\ref{def:Maxwell:weak}.
	\end{lemma}
	
	\begin{proof}
		Let $(E, H) \in X_H \times Y_H$ be a weak solution to the
		Helmholtz-type system, by Lemma
		\ref{lem:weak-Helmholtz-implies-VweakHelmholtz}, $(E, H)$ is also a
		very weak solution of the Helmholtz-type system, see
		Definiton~\ref{def:Helmholtz:Vweak}. In the following, we consider
		test-functions $\phi_i$ and $\psi_i$ for $i \in \{1,2,3\}$, which
		allow a separation of the $x_1$ variable from the $x_2$ and $x_3$
		variables, i.e. functions of the form
		\begin{align*}
			&\phi_i(x_1,x_2,x_3) = \varphi (x_1) \zeta_i(x_2,x_3) \,, \quad
			\psi_i(x_1,x_2,x_3) = \varphi (x_1) \xi_i(x_2,x_3) \quad \forall i \in \{1,2,3\}
		\end{align*}
		for arbitrary $\varphi \in D_\#(I_1)$ and arbitrary
		\begin{align*}
			&\zeta_1 \in D_\#(U\setminus \Gamma_U) \,, \,
			&&\zeta_2 \in D_\#(U) \,, \
			&&\zeta_3 \in D_\#(\overline{U}; \Gamma_U) \,, \\
			&\xi_1 \in \{\xi_1 \in D_\#(\overline{U}) \mid \partial_2 \xi_1 \in D_\#(U) \} \,, \,
			&&\xi_2 \in D_\#(\overline{U}; \Gamma_U) \,, \,
			&&\xi_3 \in D_\#(U) \,.
		\end{align*}
		
		We show that $E$ and $H$ solve \eqref{eq:Maxwell:weak} for
		test-functions of the form $\phi = \phi_i e_i$ and
		$\psi = \psi_i e_i$. By this procedure, we recover Maxwell's equations componentwise. A density argument then yields that
		\eqref{eq:Maxwell:weak} is satisfied for arbitrary $\psi \in Y$ and
		$\phi \in X$.
		
		In the following we consider arbitrary functions
		$\varphi \in D_\#(I_1)$. In the construction of test-functions, we
		use the solution $\tilde{\varphi}\in D_\#(I_1)$ of the equation
		\begin{align}\label{eq:1dHelmholtz-EV}
			- \partial_1^2 \tilde{\varphi} - \omega^2 \eps \mu
			\tilde{\varphi} = \varphi\,.
		\end{align}
		By the condition on the frequency $\omega$, this problem is uniquely
		solvable.

		\smallskip \textbf{Test-functions $\phi_1$}: Let
		$\zeta_1\in D_\#(U\setminus \Gamma_U)$ be arbitrary. We define
		$\tilde{\phi}_1(x_1,x_2,x_3) \coloneqq \tilde\fhi(x_1)
		\zeta_1(x_2,x_3)$ and, thus,
		$\tilde{\phi}_1 \in D_\#(\Omega \setminus \Gamma)$. We choose
		$\phi_1 =i \omega \eps\tilde{\phi}_1$ in
		\eqref{eq:Helmholtz:Vweak:E1}, $\psi_2 = \partial_3 \tilde{\phi}_1$
		in \eqref{eq:Helmholtz:Vweak:H2} and
		$\psi_3 = -\partial_2 \tilde{\phi}_1$ in
		\eqref{eq:Helmholtz:Vweak:H3}, which gives
		\begin{align*}	
			i \omega \eps \int_\Omega \left\{-E_1 \Delta \tilde{\phi}_1
			- \omega^2 \eps \mu E_1 \tilde{\phi}_1 \right\}
			&= \langle F_{E_1}, i \omega \eps \tilde{\phi}_1 \rangle \,,
			\\
			\int\limits_{\Omega } \left\{H_2 \partial_1^2 \partial_3 \tilde{\phi}_1
			+ \omega^2 \eps \mu H_2 \partial_3 \tilde{\phi}_1\right\}
			&= \int\limits_{\Omega } \left\{H_1 \partial_2 \partial_1 \partial_3 \tilde{\phi}_1
			+ i \omega \eps E_1 \partial_3 \partial_3 \tilde{\phi}_1\right\}
			-\langle F_{H_2}, \partial_3 \tilde{\phi}_1 \rangle \,,
			\\
			\int\limits_{\Omega } \left\{H_3 \partial_1^2 \partial_2 \tilde{\phi}_1
			+	\omega^2 \eps \mu H_3 \partial_2 \tilde{\phi}_1 \right\}
			&= \int\limits_{\Omega } \left\{H_1 \partial_3 \partial_1 \partial_2 \tilde{\phi}_1
			- i \omega \eps E_1 \partial_2 \partial_2 \tilde{\phi}_1\right\}
			- \langle F_{H_3}, \partial_2 \tilde{\phi}_1 \rangle \,.
		\end{align*}
		We subtract the third equation from the second equation, add and subtract the expression $i \omega \eps E_1 \del_1^2 \tilde{\phi}_1$ and replace
		the term $E_1 \Delta \tilde{\phi}_1$ with the first
		equation to obtain
		\begin{align}\label{eq:123456}
			\begin{aligned}
				&\int_\Omega \left\{H_2 \partial_3 (\partial_1^2\tilde{\phi}_1 +
				\omega^2\eps \mu \tilde{\phi}_1) - H_3\partial_2
				(\partial_1^2 \tilde{\phi}_1 + \omega^2\eps \mu
				\tilde{\phi}_1)\right\}
				\\
				&= \int_\Omega \left\{ - i \omega \eps E_1 (\partial_1^2
				\tilde{\phi}_1 + \omega^2\eps \mu \tilde{\phi}_1) \right\} 
				-\langle F_{H_2}, \partial_3 \tilde{\phi}_1 \rangle + \langle F_{H_3}, \partial_2 \tilde{\phi}_1 \rangle - \langle F_{E_1}, i \omega \eps \tilde{\phi}_1 \rangle
				\,.
			\end{aligned}
		\end{align}
		In order to simplify the source term, we exploit that the function satisfies $f_e\in X_0^\perp$, and, thus, is orthogonal to gradients (the product with $\nabla \partial_1 \tilde{\phi}_1$ vanishes):
		\begin{align}
			\begin{aligned}\label{eq:1234568}
			&-\langle F_{H_2}, \partial_3 \tilde{\phi}_1 \rangle + \langle F_{H_3}, \partial_2 \tilde{\phi}_1 \rangle - \langle F_{E_1}, i \omega \eps \tilde{\phi}_1 \rangle 
			\\
			&\quad =
			\int_\Omega \big\{-(f_e)_3 \partial_1 \partial_3 \tilde{\phi}_1 + i \omega \eps (f_h)_2 \partial_3 \tilde{\phi}_1
			-(f_e)_2 \partial_1 \partial_2 \tilde{\phi}_1 - i \omega \eps (f_h)_3 \partial_2 \tilde{\phi}_1
			\\
			&\qquad \qquad 
			+\omega^2 \eps \mu (f_e)_1 \tilde{\phi}_1
			- i \omega \eps f_h \cdot \curl (e_1 \tilde{\phi}_1) \big\}
			\\
			&\quad =
			\int_\Omega \big\{- f_e \cdot \nabla \partial_1 \tilde{\phi}_1 + (f_e)_1 \partial_1^2 \tilde{\phi}_1
			+ \omega^2 \eps \mu (f_e)_1 \tilde{\phi}_1 \big\}
			\\
			&\quad =
			\int_\Omega \big\{(f_e)_1 \partial_1^2 \tilde{\phi}_1
			+ \omega^2 \eps \mu (f_e)_1 \tilde{\phi}_1 \big\}\,.
			\end{aligned}
		\end{align}
		Replacing the source terms in \eqref{eq:123456} with \eqref{eq:1234568}, we obtain
		\begin{align}\label{eq:proof:1}
	\begin{aligned}
		&\int_\Omega \left\{H_2 \partial_3 (\partial_1^2\tilde{\phi}_1 +
		\omega^2\eps \mu \tilde{\phi}_1) - H_3\partial_2
		(\partial_1^2 \tilde{\phi}_1 + \omega^2\eps \mu
		\tilde{\phi}_1)\right\}
		\\
		&\quad = \int_\Omega \left\{ - i \omega \eps E_1 (\partial_1^2
		\tilde{\phi}_1 + \omega^2\eps \mu \tilde{\phi}_1) + (f_e)_1
		(\partial_1^2 \tilde{\phi}_1 + \omega^2\eps \mu
		\tilde{\phi}_1) \right\}\,,
	\end{aligned}
\end{align}
		which is a weak formulation of 
				\begin{equation*}
		(\partial_1^2 + \omega^2 \eps \mu) (\curl H)_1 
		= (\partial_1^2 + \omega^2 \eps \mu) (-i \omega \eps E_1 + (f_e)_1)\,.
		\end{equation*} 
		Because of \eqref{eq:1dHelmholtz-EV}, the function
		$\phi_1(x_1,x_2,x_3)\coloneqq \varphi(x_1)\zeta_1(x_2,x_3)$
		satisfies
		$\phi_1 =- \partial_1^2 \tilde{\phi}_1 - \omega^2 \eps \mu
		\tilde{\phi}_1$. This allows to simplify \eqref{eq:proof:1} to
		\begin{align*}
			\int_\Omega \left\{H_2 \partial_3 {\phi}_1 - H_3\partial_2 {\phi}_1
			+
			i \omega \eps E_1 {\phi}_1\right\} = \int_\Omega (f_e)_1 {\phi}_1 \,.
		\end{align*}
		This is relation \eqref{eq:Maxwell:weak:H} for $\phi = e_1
		\phi_1$. We recall that we have chosen
		$\phi_1(x_1,x_2,x_3) = \varphi(x_1)\zeta_1(x_2,x_3)$ with arbitrary
		$\varphi \in D_\#(I_1)$ and $\zeta_1\in D_\#(U\setminus \Gamma_U)$.

		\smallskip \textbf{Test-functions $\psi_1$}: We proceed as with
		$\phi_1$. Let $\xi_1 \in D_\#(\overline{U})$ with
		$\partial_2 \xi_1 \in D_\#(U)$. We define
		$\tilde{\psi}_1(x_1,x_2,x_3) \coloneqq \tilde{\varphi}(x_1)
		\xi_1(x_2, x_3)$ such that
		$\tilde{\psi}_1 \in D_\#(\overline{\Omega})$ and
		$\partial_3 \tilde{\psi}(x_1,x_2,x_3) = \tilde{\varphi}(x_1)
		\partial_2 \xi_1(x_2, x_3)$ and
		$\partial_3 \tilde{\psi} \in D_\#(\Omega)$. We choose
		$\psi_1 = -i \omega \mu \tilde{\psi}_1$ in
		\eqref{eq:Helmholtz:Vweak:H1},
		$\phi_2 = \partial_3 \tilde{\psi}_1 \in D_\#(\Omega)$ in
		\eqref{eq:Helmholtz:Vweak:E2} and
		$\phi_3 = -\partial_2 \tilde{\psi}_1$ in
		\eqref{eq:Helmholtz:Vweak:E3}. Adding the resulting equations and exploiting that the function $f_h$ is
		orthogonal to gradients shows
		\eqref{eq:Maxwell:weak:E} for $\psi = e_1 \psi_1$ with
		$\psi_1(x_1,x_2,x_3) = \varphi(x_1)\xi_1(x_2,x_3)$.
		
		\smallskip\textbf{Test-functions $\phi_2$}: For
		$\zeta_2 \in D_\#(U)$, we use
		$\tilde{\phi}_2(x_1,x_2,x_3) =
		\tilde{\varphi}(x_1)\zeta_2(x_2,x_3)$, which satisfies
		$\tilde{\phi}_2 \in D_\#(\Omega)$. We can choose
		$\phi_2 = i \omega \eps \tilde{\phi}_2$ in
		\eqref{eq:Helmholtz:Vweak:E2} and
		$\psi_3 = \partial_1 \tilde{\phi}_2 \in D_\#(\Omega)$ in
		\eqref{eq:Helmholtz:Vweak:H3}. Adding the resulting equations yields
		\begin{align}\label{eq:proof:2}
			\begin{aligned}
				\int_{\Omega } \left\{- H_1 \partial_3 (-\partial_1
				\partial_1 \tilde{\phi}_2 - \omega^2 \eps
				\mu\tilde{\phi}_2) + H_3 \partial_1(-\partial_1 \partial_1
				\tilde{\phi}_2 - \omega^2 \eps \mu\tilde{\phi}_2) \right\}
				\\
				= \int_{\Omega } \left(-i \omega \eps E_2 +(f_e)_2\right)
				(-\partial_1 \partial_1 \tilde{\phi}_2 - \omega^2 \eps
				\mu\tilde{\phi}_2)\,.
			\end{aligned}
		\end{align}
		We simplify again with
		${\phi}_2(x_1,x_2,x_3) = {\varphi}(x_1)\zeta_2(x_2,x_3) =-
		\partial_1^2 \tilde{\phi}_2 - \omega^2 \eps \mu \tilde{\phi}_2$ to
		\begin{align*}
			\int_{\Omega } \left\{- H_1 \partial_3 {\phi}_2 + H_3 \partial_1\phi_2 \right\}
			= \int_{\Omega } \left(-i \omega \eps E_2 +(f_e)_2\right) \phi_2\,.
		\end{align*}
		This is \eqref{eq:Maxwell:weak:H} for $\phi = e_2 \phi_2$.
		
		\smallskip \textbf{Test-functions $\psi_2$}: For a function
		$\xi_2 \in D_\#(\overline{U};\Gamma_U)$ we use
		$\tilde{\psi}_2(x_1,x_2,x_3) \coloneqq \tilde{\varphi}(x_1)
		\xi_2(x_2,x_3)$, which satisfies
		$\tilde{\psi}_2 \in D_\#(\overline{\Omega}; \Gamma)$. We can
		choose $\psi_2 = -i \omega \mu \tilde{\psi}_2$ in
		\eqref{eq:Helmholtz:Vweak:H2} and
		$\phi_3 = \partial_1 \tilde{\psi}_2$ in
		\eqref{eq:Helmholtz:Vweak:E3}. Adding the resulting equations yields
		\eqref{eq:Maxwell:weak:E} for $\psi = e_2 \psi_2$ with
		$\psi_2(x_1,x_2,x_3) = \varphi(x_1) \xi_2(x_2,x_3)$.

		\smallskip \textbf{Test-functions $\phi_3$}: For a function
		$\zeta_3 \in D_\#(\overline{U};\Gamma_U)$ we use
		$\tilde{\phi}_3(x_1,x_2,x_3) \coloneqq
		\tilde{\varphi}(x_1)\zeta_3(x_2,x_3)$, which satisfies
		$\tilde{\phi} \in D_\#(\overline{\Omega}, \Gamma)$. We can choose
		$\phi_3 = i \omega \eps \tilde{\phi}_3$ in
		\eqref{eq:Helmholtz:Vweak:E3} and
		$\psi_2 = \partial_1 \tilde{\phi}_3$ in
		\eqref{eq:Helmholtz:Vweak:H2}. Adding the resulting equations yields
		\eqref{eq:Maxwell:weak:H} for $\phi = e_3 \phi_3$ with
		$\phi_3(x_1,x_2,x_3) = \varphi(x_1)\zeta_3(x_2,x_3)$.
		
		\smallskip\textbf{Test-functions $\psi_3$}: For $\xi_3 \in D_\#(U)$
		we use
		$\tilde{\psi}_3(x_1,x_2,x_3) = \tilde{\varphi}(x_1) \xi_3(x_2,x_3)$,
		which satisfies $\tilde{\psi}_3\in D_\#(\Omega)$. We can choose
		$\psi_3 = -i \omega \mu \tilde{\psi}_3$ in
		\eqref{eq:Helmholtz:Vweak:H3} and
		$\phi_2 = \partial_1 \tilde{\psi}_3$ in
		\eqref{eq:Helmholtz:Vweak:E2}. Adding the resulting equations
		yields \eqref{eq:Maxwell:weak:E} for $\psi = e_3 \psi_3$ with
		$\psi_3(x_1,x_2,x_3) = \varphi(x_1) \xi_3(x_2,x_3) $.
		
		\smallskip \textbf{Density of test-functions:} Taking finite linear
		combinations, we obtain that $(E,H)$ solves \eqref{eq:Maxwell:weak}
		for arbitrary $\phi \in \calX$ and $\psi \in \calY$, where $\calX$
		and $\calY$ are given as vector spaces of smooth functions. In the
		subsequent formula, the index $i \in \{1,2,3\}$ stands for the
		components of the vector fields $\phi = (\phi_1,\phi_2,\phi_3)$ and
		$\psi = (\psi_1,\psi_2,\psi_3)$:
		\begin{equation}\label{eq:def:AB}
			\begin{split}
				&\calX \coloneqq \Bigg\{\phi \colon \Omega \to \C^3 \,\bigg|
				\,\phi_i(x_1,x_2,x_3) = \sum\limits_{j = 1}^N
				\varphi_i^{j}(x_1)\zeta_i^{j}(x_2,x_3) \text{ for some $N\in \N$},\,
				\\
				&\hspace{15mm} \varphi_i^{j} \in D_\#(I_1)\, \forall i\leq 3\,,\zeta_1^{j} \in D_\#(U\setminus \Gamma_U),
				\zeta_2^{j} \in D_\#(U), \zeta_3^{j} \in D_\#(\overline{U};
				\Gamma_U)\ \forall j \Bigg\}\,,
				\\
				&\calY \coloneqq \Bigg\{\psi \colon \Omega \to \C^3 \,\bigg|\,
				\psi_i(x_1,x_2,x_3) = \sum\limits_{j = 1}^N
				\varphi_i^{j}(x_1)\xi_i^{j}(x_2,x_3) \text{ for some $N \in \N$},\,
				\\
				&\hspace{5mm}	\varphi_i^{j} \in D_\#(I_1)\, \forall i\,,\,\xi_1^{j} \in D_\#(\overline{U})\,,\,
				\partial_{x_3} \xi_1^j \in D_\#(U), \xi_2^{j} \in D_\#(\overline{U},
				\Gamma_U),
				\xi_3^{j} \in D_\#(U)\forall j\Bigg\}\,.
			\end{split}
		\end{equation}
		By Lemma \ref{lem:density} below, the set $\calX$ is dense in $X$ (in the
		topology of $H(\curl, \Omega)$), and $\calY$ is dense in $Y$ (in the
		topology of $H(\curl, \Omega\setminus \Gamma)$). We therefore obtain
		\eqref{eq:Maxwell:weak} for all $\psi \in Y$ and $\phi \in X$. This
		verifies that $(E,H)$ is a solution to the Maxwell system.
	\end{proof}
	
	It remains to prove the density result of Lemma \ref{lem:density}. As
	a preparation, we show a density property for scalar-valued functions,
	formulated as Lemma~\ref {lem:density-vanishing-normal-derivative}. We
	mention that a similar argument is also used in the proof of
	Lemma~\ref {lem:regularity:Distrib-Boundary-Neumann}.
	\begin{lemma}[$H^1$-density of functions with vanishing normal derivative]
		\label{lem:density-vanishing-normal-derivative} We define a vector
		space of scalar-valued smooth functions as
		
		\begin{align*}
			\calZ \coloneqq \Big\{u \in
			& D_\#(\overline{\Omega}) 
			\, \Big| \,
			u(x) = \sum\limits_{j = 1}^N
			\phi^{j}_1(x_1) \phi^{j}_2(x_2) \zeta^{j}_3(x_3)\text{ for some $N\in \N$} ,\,	 \\
			&\phi_1^{j} \in D_\#(I_1) \,, \, \phi_2^{j} \in D_\#(I_2)\,,\, 
			\zeta^{j} \in C^\infty(\overline{I_3})
			\text{ with } \partial_{x_3}\zeta^{j} \in C^\infty_c(I_3)\,
			\forall j\Big\}\,.
		\end{align*}
		This space is dense in $H^1_\#(\Omega)$, there holds
		$\overline{\calZ}^{H^1(\Omega)} =H^1_\#(\Omega)$
	\end{lemma}
	
	\begin{proof}
		Let $u \in H^1_\#(\Omega)$ be an arbitrary function that we wish to
		approximate and let $\delta>0$ be arbitrary. 
		We choose an
		approximation
		$u_\delta(x) = \sum\limits_{j=1}^{N_\delta} \phi_\delta^j(x_1,x_2)
		\zeta_\delta^j(x_3)$ for some $N_\delta\in\N$ with
		$\|u - u_\delta\|_{H^1(\Omega)} < \delta$ with the properties:
		$\phi_\delta^j(x_1,x_2) = \phi_{\delta,1}^j(x_1) \phi_{\delta,2}^j$ for
		$\phi_{\delta,1}^{j} \in D_\#(I_1)$, $\phi_{\delta,2}^{j} \in D_\#(I_2)$
		and $\zeta_\delta^j \in C^\infty(\overline{I_3})$ for all
		$j \le N_\delta$. The approximation $u_\delta$ is not necessarily in
		$\calZ$, since $\partial_{x_3}\zeta^{j} \in C^\infty_c(I_3)$ is not
		guaranteed.
		
		We therefore construct a second approximation: For fixed $\delta$, we
		can approximate, for every $j \le N_\delta$, the derivative
		$\partial_{x_3} \zeta_\delta^j$ by a function
		$f_\delta^j \in C^\infty_c(I_3)$ such that
		$\|\partial_{x_3} \zeta_\delta^j -f_\delta^j\|_{L^2(I_3)} < \delta
		N_\delta^{-1} C_\delta^{-1}$ for
		$C_\delta \coloneqq \max\limits_{j \in \{1, \dots N_\delta\}}
		\|\phi_\delta^j\|_{C^1(\overline{\Omega})}$. We define an improved
		approximating function $\zeta$ as the integral over $f$ as follows:
		\begin{align*}
			\tilde{\zeta}_\delta^j(x_3) \coloneqq \tilde{\zeta}_\delta^j(l_3)
			+ \int_{l_3}^{x_3}f_\delta^j(t) \mathrm{d} t\,.
		\end{align*}
		This approximation satisfies, for some constant $C$ that is
		independent of $\delta$, $j$, $N_\delta$ and $C_\delta$, an estimate
		$\|\zeta_\delta^j- \tilde{\zeta}_\delta^j\|_{H^1(I_3)} < \delta C
		N_\delta^{-1} C_\delta^{-1}$. We can now define the approximation
		$\tilde{u}_\delta \in \calZ$ as
		\begin{align*}
			\tilde{u}_\delta(x) \coloneqq
			\sum\limits_{j=1}^{N_\delta} \phi_\delta^j(x_1,x_2) \tilde{\zeta}_\delta^j(x_3) \,.
		\end{align*}
		We obtain that $\tilde{u}_\delta$ is a good approximation of $u$ from
		\begin{align*}
			\|u_\delta -\tilde{u}_\delta\|_{L^2(\Omega)} 
			&\leq \sum\limits_{j=1}^{N_\delta} \|\phi_\delta^j(x_1,x_2)
			(\zeta_\delta^j - \tilde{\zeta}_\delta^j) \|_{L^2(\Omega)}
			\\
			&\leq N_\delta \max\limits_{j \le N_\delta}
			\|\phi_\delta^j\|_{C^0(\overline{\Omega})}
			\max\limits_{j \le N_\delta} \|\zeta_\delta^j - \tilde{\zeta}_\delta^j\|_{L^2(\Omega)}
			\leq C \delta\,.
		\end{align*}
		Derivative such as
		$\|\partial_i (u_\delta -\tilde{u}_\delta)\|_{L^2(\Omega)}$ for
		$i \in \{1,2,3\}$ are also of order $\delta$, as can be shown with
		the same calculation, replacing functions that depend on $x_i$ by
		their derivative with respect to $x_i$.
	\end{proof}

	\begin{lemma}[Density of functions in $X$ and $Y$]
		\label{lem:density}
		The vector spaces $\calX,\calY$ of \eqref{eq:def:AB} have the
		density properties $\overline{\calX}^{H(\curl, \Omega)} = X$ and
		$\overline{\calY}^{H(\curl, \Omega\setminus \Gamma)} = Y$.
	\end{lemma}
	
	Before we start the slightly technical proof, let us sketch the
	overall idea. Let us assume that we want to approximate a function
	$H\in Y$ up to an error of order $\delta$. We construct an
	approximation $\tilde H_\delta$ in two steps. Step 1: With a
	convolution, we mollify $H$ in the horizontal variables $x_1$ and
	$x_2$; this defines the first approximation, $H_\delta$. The
	approximation keeps relevant properties of $H$ and satisfies,
	additionally, smoothness,
	$H_{\delta,1}, H_{\delta,2} \in H^1(\Omega\setminus\Gamma)$, $H_{\delta,3} \in L^2(I_3, H^1_\#(I_1 \times I_2, \C^3))$ (we recall
	that, for the original function $H$, we have only a control of the
	curl). Step 2: We approximate $H_\delta$ to satisfy all desired
	properties.

	\begin{proof}
		\textit{Approximation of $H \in Y$:} Let $H \in Y$ be an arbitrary
		function that we wish to approximate. {\em Step 1: Horizontal
			smoothing.} For a compactly supported function $\eta$ on $\R^2$
		with integral $1$, we define the standard mollifier as
		$\eta_\delta \coloneqq \delta^{-2} \eta(\cdot / \delta)$ for
		$\delta >0$. We choose
		$H_\delta \coloneqq H \ast_{1,2} \eta_\delta$, where $\ast_{1,2}$
		denotes the convolution with respect to the first and second
		argument; with the notation $y = (\hat{y}, 0)$ for
		$\hat{y} \in \R^2$, we set
		\begin{align*}
			H_\delta(x) \coloneqq \int_{\R^2} \curl H(x-y)\, \eta_\delta(y)
			\,\mathrm{d}\hat{y} \,.
		\end{align*}
		The approximation is smooth in tangetial directions, in particular,
		there holds $H_\delta \in L^2(I_3, H^1_\#(I_1 \times I_2, \C^3))$.
		Since the convolution commutes with derivatives, we also obtain
		$\curl H_\delta = \curl (H \ast_{1,2} \eta_\delta) = \curl
		H\ast_{1,2} \eta_\delta \in L^2(\Omega_\pm)$ and, thus,
		$H_\delta \to H$ in $H(\curl, \Omega \setminus \Gamma)$. A standard
		calculation yields, for arbitrary $\phi =(\phi_1, \phi_2, \phi_3)$
		with
		$\phi_1 \in D_\#(\Omega\setminus\Gamma)\,,\phi_2 \in D_\#(\Omega)\,,
		\phi_3 \in D_\#(\Omega)$, that
		\begin{align}\label{eq:approximation-H}
			\begin{aligned}
				\int_{\Omega\setminus \Gamma} \curl H_\delta \cdot \phi =
				\int_{\Omega\setminus \Gamma} H_\delta \cdot \curl\phi \,.
			\end{aligned}
		\end{align}
		This shows that the first component of $H_\delta$ does not jump
		across $\Gamma$, the approximation satisfies $H_\delta \in Y$. The
		main point of this proof is that the control of the curl now allows
		to control all derivatives, at least for the first two components:
		The distributional derivatives in $\Omega\setminus \Gamma$ are
		\begin{align}
			\label{eq:derivatives-Hdelta-1}
			\partial_3 H_{\delta,1} &= (\curl H_\delta)_2 + \partial_1 H_{\delta,3}\,,\\
			\label{eq:derivatives-Hdelta-2}
			\partial_3 H_{\delta,2} &= -(\curl H_\delta)_1 - \partial_2 H_{\delta,3}\,,
		\end{align}
		and are thus elements of $L^2(\Omega)$.
		This yields
		$H_{\delta,1}, H_{\delta,2} \in H^1_\#(\Omega\setminus \Gamma)$.
		Since the first component of $H_\delta$ does not jump across
		$\Gamma$, we even have $H_{\delta,1}\in H^1_\#(\Omega)$.
		
		\smallskip {\em Step 2: Second approximation.} By classical
		approximation arguments and by Lemma
		\ref{lem:density-vanishing-normal-derivative}, we can choose a
		smooth approximation $\tilde{H}_{\delta} \in \calX$ of $H_\delta$
		such that
		\begin{align}\label{eq:approximation-H-d}
			\|\tilde{H}_{\delta,1} -H_{\delta,1}\|_{H^1(\Omega)}
			+ \|\tilde{H}_{\delta,2} -H_{\delta,2}\|_{H^1(\Omega\setminus \Gamma)}+
			\|\tilde{H}_{\delta,3} -H_{\delta,3}\|_{L^2(I_3, H^1_\#(I_1 \times I_2)} \leq \delta\,.
		\end{align}
		The proof is complete when we show that this approximation also
		satisfies that
		$\|\tilde{H}_{\delta} - H_\delta\|_{H(\curl, \Omega\setminus
			\Gamma)}$ is of order $\delta$. For the
		$L^2(\Omega\setminus \Gamma)$-norm, this is obvious from \eqref
		{eq:approximation-H-d}. We estimate the curl as
		\begin{align}\label{eq:approximation-curl-H-d}
			\| \curl (\tilde{H}_\delta - H_\delta) \|_{L^2(\Omega\setminus \Gamma)}
			= \left\|
			\dvec{\partial_2 (\tilde{H}_{\delta,3} - {H}_{\delta,3}) - \partial_3 (\tilde{H}_{\delta,2} - H_{\delta,2})}
			{\partial_3 (\tilde{H}_{\delta,1} - {H}_{\delta,1}) - \partial_1 (\tilde{H}_{\delta,3} - H_{\delta,3})}
			{\partial_1 (\tilde{H}_{\delta,2} - {H}_{\delta,2}) - \partial_2 (\tilde{H}_{\delta,1} - H_{\delta,1})} \right\|_{L^2(\Omega\setminus \Gamma)} 
			\leq 6 \delta\,,
		\end{align}
		where we used the $L^2(\Omega)$-smallness
		\eqref{eq:approximation-H-d} of the entries.
		
		\textit{Approximation of $E \in X$:} We approximate $E \in X$
		analogously to $H \in Y$. We define
		$E_\delta \coloneqq E \ast_{1,2} \eta_\delta \in H(\curl, \Omega)$.
		One obtains the relation of \eqref{eq:approximation-H} with
		$H_\delta$ replaced $E_\delta$, and with $\phi$ replaced by
		$\psi \in Y$. This verifies $E_\delta \in X$. We identify the
		distributional derivatives of $E$ by means of the curl similarly to
		\eqref{eq:derivatives-Hdelta-1}--\eqref{eq:derivatives-Hdelta-2} and
		obtain $E_{\delta,1}, E_{\delta,2} \in H^1_\#(\Omega)$. Moreover,
		because of $E \in X$, we can deduce further
		$E_{\delta,1} \in H^1_{0,\#}(\Omega\setminus \Gamma)$ and
		$E_{\delta,2} \in H^1_{0,\#}(\Omega)$. Finally, we approximate
		$E_\delta$ by $\tilde{E}_\delta \in \calX$ with respect to the
		$H^1(\Omega) \times H^1(\Omega\setminus \Gamma) \times L^2(I_3,
		H^1_\#(I_1 \times I_2))$-norm by classical approximation arguments.
		The convergence in $H(\curl, \Omega)$ follows as in
		\eqref{eq:approximation-curl-H-d}.
	\end{proof}

	\section{Existence and uniqueness of the Helmholtz-type system}
	\label{sec.ex-un-Helmholtz}

	\begin{proposition}[Existence and uniqueness for Helmholtz]
		\label{prop:existence-Htz}
		Let $\Omega$ and $\Gamma$ be as in \eqref{eq:def:Omega}, we consider
		the Helmholtz system \eqref{eq:Helmholtz:strong} with parameters
		$\eps, \mu >0$ and a frequency $\omega$ with
		$\omega^2 \notin \sigma_M$ of \eqref {eq:spec-M}. Then, for
		arbitrary $(f_h,f_e) \in L^2(\Omega, \C^3)^2$, system
		\eqref{eq:Helmholtz:strong} has a unique weak solution $(E, H)\in X_H \times Y_H$.
	\end{proposition}

	\begin{proof}
		Because of $\omega^2 \notin \sigma_M$, the two Helmholtz problems
		\eqref {eq:Helmholtz:strong:E1} for $E_1$ in $\Omega_+$ and in
		$\Omega_-$ can be solved uniquely. Furthermore, the Helmholtz
		problem \eqref{eq:Helmholtz:strong:H1} for $H_1$ in $\Omega$ can be
		solved uniquely.
		
		The operator $(-\partial_1^2 - \omega^2 \eps \mu \operatorname{id})$
		can be inverted. More precisely, the equation
		\begin{equation}\label{eq:1D}
			-\partial_1^2 \varphi - \omega^2 \eps \mu \varphi = g + \del_1 h 
		\end{equation}
		has a unique solution $\fhi\in H^1_\#(I_1)$ for every
		$g, h\in L^2(I_1)$, as can be shown, e.g., with Fourier series. This
		observation allows us to solve the remaining four equations of \eqref
		{eq:Helmholtz:strong} uniquely for $E_2, H_2, E_3, H_3$.
	\end{proof}
	
	The next lemma shows that the spectrum $\sigma(H)$ of
	\eqref{eq:Helmholtz:weak} is $\sigma_M$.
	
	\begin{lemma}[Spectrum of the Helmholtz-type system]
		\label{lem:Existence:Helmholtz} In the situation of
		Proposition~\ref{prop:existence-Htz}, but with
		$\omega^2 \in \sigma_M$, the homogeneous Helmholtz-type system has a
		non-trivial solution.
	\end{lemma}
	
	\begin{proof}
		In Lemma~\ref{lem:eigenfct-Max}, we obtained a non-trivial solution
		$(E,H)$ to the homogeneous Maxwell system. By Lemma
		\ref{lem:Maxwell-implies-Helmholtz} this solution is also a
		non-trivial solution $(E,H)$ to the homogeneous Helmholtz-type
		system.
	\end{proof}
	
	\begin{proof}[Proof of Theorem \ref{thm:spectrum}]
		For $\omega^2 \notin \sigma_M$, Proposition~\ref
		{prop:existence-Htz} yields the existence of a solution to the
		Helmholtz system. By Lemma~\ref {lem:Helmholtz-implies-Maxwell},
		this solution solves also the Maxwell system. Regarding uniqueness:
		When $(E,H)$ is a solution to the homogeneous Maxwell system, then,
		by Lemma~\ref {lem:Maxwell-implies-Helmholtz}, it is also a solution
		to the homogeneous Helmholtz system, and hence vanishes by the
		uniqueness statement of Proposition~\ref {prop:existence-Htz}.
		
		For $\omega^2 \in \sigma_M$, non-trivial solutions to the
		homogeneous Maxwell system are provided in Lemma~\ref
		{lem:eigenfct-Max}.
	\end{proof}

	\appendix
	
	\section{Regularity of distributional solutions}
	
	In order to prove that very weak solutions are also weak solutions,
	one has to show that very weak solutions have an $H^1$-type
	regularity. A classical result of this type is known as Weyl's lemma:
	\begin{lemma}[Lemma of Weyl for $H^{-1}$-right-hand sides]\label{lem:Weyl-H1}
		Let $U \subset \R^n$ be open and bounded, $n \in \N$.
		Let $u \in L^1_\loc(U,\C)$ be a	distributional solution of $-\Delta u = f - \nabla\cdot g$ for data $f \in L^2(U,\C)$ and $g \in L^2(U,\C^n)$ in the sense that
		\begin{align}\label{eq:dist:Laplace}
			-\int_U u \Delta \phi 
			&= \int_U f \phi 
			+ g \cdot \nabla \phi && \forall \phi \in C_c^\infty(U)\,.
		\end{align}
		Then, $u$ has the regularity $u \in H^1_\loc(U)$.
	\end{lemma}
	\begin{proof}
		Let $u \in L^1_\loc(U)$ be a solution to \eqref{eq:dist:Laplace}. We consider
		the weak solution $\tilde{u} \in H^1_0(U,\C)$ of $\int_U \nabla \tilde{u} \cdot \nabla \phi = \int_U f \phi + g \cdot \nabla \phi$ $\forall \phi \in H^1_0(U)$.
		The difference $(\tilde{u}- u)$ is harmonic in the sense that $\int_U (\tilde{u}- u) \Delta \phi =0$ for all $\phi \in C_c^\infty(U)$. Thus, by the classical Weyl lemma, $(\tilde{u}- u) \in C^\infty(U)$, see \cite{Weyl-1940}. This implies $u\in H^1_\loc(U)$.
	\end{proof}
	
	We note that Lemma \ref{lem:Weyl-H1} provides only local
	$H^1$-regularity. In order to obtain $H^1$-regularity on the entire
	domain, we extend the function in all directions and use the local
	$H^1$-regularity of the extended function to conclude. Extensions over
	the periodic boundary are directly given with the periodic extension
	of the function (compare the definition of periodic functions).
	Across Dirichlet boundaries, we use an odd extension, see
	Lemma~\ref{lem:regularity:Distrib-Boundary-Dirichlet}. Across Neumann
	boundaries, we use an even extension, see
	Lemma~\ref{lem:regularity:Distrib-Boundary-Neumann}. The subsequent
	lemma is used in the text for the domains $\Omega_+$ and $\Omega_-$.

	\begin{lemma}[Regularity for vanishing distributional Dirichlet
		boundary data]
		\label{lem:regularity:Distrib-Boundary-Dirichlet}
		Let the domain be a product of three open intervals,
		$\Omega = I_1\times I_2\times I_3 \subset \R^3$. For right-hand
		sides $f \in L^2(\Omega,\C)$ and $g \in L^2(\Omega,\C^3)$, let
		$u \in L^2(\Omega,\C)$ satisfy
		\begin{align}\label{eq:Helmholtz-distributional-Dirichlet}
			- \int\limits_\Omega u \Delta \phi 
			&= \int\limits_\Omega f \phi + g\cdot \nabla \phi
			&& \forall \phi \in D_\#(\Omega) \,.
		\end{align}	
		On the upper and lower boundary, let $u$ satisfy a Dirichlet
		condition in the following sense: For every
		$\fhi \in C^{\infty}_\#(I_1 \times I_2) $ there exists a function
		$h = h_\fhi\in L^2(I_3)$ such that
		\begin{align}
			\label{eq:vanish:distributional:Dirichlet}
			&\int_{I_3} \left\{ \left(\int_{I_1 \times I_2}
			u\, \fhi\right)\, \partial_3 \psi
			+ h \psi\right\} = 0 
			&& \forall \psi \in C^\infty(\overline{I_3})\,.
		\end{align}	
		Then, the solution $u$ of
		\eqref{eq:Helmholtz-distributional-Dirichlet} is of class
		$u \in H^1_{0,\#}(\Omega)$ and is a weak solution in the sense that
		\begin{align}\label{eq:IntParts-Dirichlet}
			\int\limits_\Omega \nabla u \cdot \nabla \phi
			&= \int\limits_\Omega f \phi + g \cdot \nabla \phi
			&& \forall \phi \in H^1_{0,\#}(\Omega) \,.
		\end{align}
	\end{lemma}

	Regarding the boundary condition
	\eqref{eq:vanish:distributional:Dirichlet}: The condition implies
	that, for every test-function $\fhi$, the expression
	$G_\fhi \colon x_3\mapsto \left(\int_{I_1 \times I_2}u \fhi\right)$ has a
	weak derivative in $L^2(I_3)$, namely $h = \partial_3 G_\fhi$. This
	implies not only $G_\fhi\in H^1(I_3)$, but it also yields that the
	boundary values vanish, $G_\fhi\in H^1_0(I_3)$.
	
	\begin{proof}
		As announced, we extend $u$ in all directions and conclude by
		applying the Lemma of Weyl to the extended function. Without loss
		of generality, we assume $I_3 = (0,l_3)$. In directions $x_1$ and
		$x_2$, we extend periodically; to simplify notation, we keep the
		notation $u$ for the extended function,
		$u = \tilde{u} \in H^1_\loc(\Omega_\#)$. In order to show the
		$H^1$-regularity up to the boundaries $x_3 = l_3$ and $x_3 = 0$. We
		perform the proof for the lower boundary $x_3 = 0$, the other proof
		is analogous. We set
		$\check\Omega \colon = I_1\times I_2\times (-l_3,l_3)$. We reflect $u$ odd
		across the boundary $x_3 = 0$ to define the extended function
		$\check{u}$. When we show that the distributional Laplace of the
		extended function is of class $H^{-1}(\check\Omega)$, we can use the
		Lemma of Weyl on the extended domain to conclude the
		$H^1$-regularity of the original function up to the boundary.
		
		\smallskip For $x \in \Omega$ with $(x_1,x_2) \eqqcolon \hat{x}$, we
		define $\check{u}(\hat{x}, x_3)\in L^2(\Omega)$ by
		\begin{equation*}
			\check{u}(\hat{x}, x_3) \coloneqq \begin{cases} u(\hat{x}, x_3) &
				\text{if } x_3 >0\,,
				\\
				-u(\hat{x}, -x_3) & \text{if } x_3 <0 \,,
			\end{cases}
		\end{equation*}
		the functions $\check{f}$ and $\check{g}$ in the same way as odd
		extensions.
		
		We consider test-functions as follows: The horizontal dependence is
		given by $\varphi \in D_\#(I_1 \times I_2)$. Vertical dependences
		are given by $\psi \in C^\infty_c((-l_3,l_3))$. We define
		$\check{\psi}(x_3) = \psi(x_3)- \psi(-x_3)$, which is a smooth odd
		function with, in particular, $\check{\psi}(0)= 0$. This allows to
		find a sequence $\check{\psi}_n \in C^\infty_c(I_3)$ such that
		$\check{\psi}_n \to \check{\psi}|_{I_3}$ in $H^1(I_3)$.
		
		With the aim to show that $\check{u}$ solves an elliptic equation on
		$\check\Omega$, we compute, using the notation
		$\Delta_{1, 2} = \del_1^2 + \del_2^2$,
		\begin{align*}
			&\int\limits_{\check\Omega} \check{u} \Delta (\varphi\psi) = 
			\int\limits_{\Omega} \check{u} \Delta (\varphi\psi)
			+
			\int\limits_{\check\Omega\setminus\Omega} \check{u} \Delta (\varphi\psi)\\
			&=
			\int\limits_{\Omega} \check{u} \Delta (\varphi\check \psi)
			= \int\limits_{\Omega} u\, (\Delta_{1, 2} \varphi\, \check{\psi}
			+ \varphi\, \partial_3^2 \check{\psi}) 
			\\
			&\overset{\eqref{eq:vanish:distributional:Dirichlet}}{=}
			\int\limits_{\Omega} u\, \Delta_{1, 2} \varphi\, \check{\psi} 	 
			- \int\limits_{I_3}\partial_3\left( \,
			\int\limits_{I_1 \times I_2} u\,\varphi \right) \partial_{x_3} \check{\psi} 
			\\
			&= \lim\limits_{n \to \infty}
			\int\limits_{\Omega} u\, \Delta_{1, 2} \varphi\, \check{\psi}_n 
			- \int\limits_{I_3}\partial_3\left( \, \int\limits_{I_1 \times I_2} u\varphi \right)
			\partial_{x_3} \check{\psi}_n 
			\\
			&= \lim\limits_{n \to \infty}
			\int\limits_{\Omega} u\, \Delta_{1, 2} \varphi\, \check{\psi}_n 
			+ \int\limits_{I_3}\left( \, \int\limits_{I_1 \times I_2} u\varphi \right) \partial_{x_3}^2 \check{\psi}_n 
			\\
			&= \lim\limits_{n \to \infty} \int\limits_{\Omega} u\, \Delta (\varphi \check{\psi}_n)\\
			&\overset{\eqref{eq:Helmholtz-distributional-Dirichlet}}{=}
			\lim\limits_{n \to \infty} \int\limits_{\Omega}
			f\, \varphi\, \check{\psi}_n + g \cdot \nabla \varphi\, \check{\psi}_n
			\\
			&	=
			\int\limits_{\Omega} f\,\varphi\, \check{\psi}
			+ g \cdot \nabla \varphi\, \check{\psi}
			= 
			\int\limits_{\check\Omega} \check{f}\,
			\varphi\,\psi + \check{g} \cdot \nabla \varphi\, \psi\,.
		\end{align*}
		Since test-functions of the form $\phi = \fhi(x_1,x_2)\psi(x_3)$
		span a dense (dense with respect to the norm of $H^2(\check\Omega)$)
		subset of all smooth functions, we obtain
		\begin{align}
			-\int\limits_{\check \Omega} \check{u}\, \Delta \phi = 
			\int\limits_{\check \Omega} \check{f}\, \phi + \check{g} \cdot \nabla \phi\,.
		\end{align}
		for all $\phi \in D_\#(\check\Omega)$. As announced, Lemma
		\ref{lem:Weyl-H1} allows us to conclude
		$u \in H^1_{\#, \loc}(\check\Omega)$. Using an analogous reflection
		argument at the top boundary, we can deduce the $H^1$-regularity up
		to the boundary, i.e.~$u \in H^1_{\#}(\Omega)$.
		
		\smallskip It remains to check the Dirichlet condition in the sense
		of traces, $u \in H^1_{0,\#}(\Omega)$. Because of
		$u \in H^1_\#(\Omega)$, we can consider the traces on, e.g., the
		boundary $x_3 = 0$. Relation
		\eqref{eq:vanish:distributional:Dirichlet} yields, for every
		function $\fhi\in C^\infty_\#(I_1 \times I_2)$, that
		$\int_{I_1 \times I_2} u \fhi$ vanishes along $x_3 = 0$. This
		implies that the trace of $u$ along this boundary vanishes, we
		obtain $u \in H^1_{0,\#}(\Omega)$.
	\end{proof}

	\begin{lemma}[Regularity for vanishing distributional Neumann boundary
		data]
		\label{lem:regularity:Distrib-Boundary-Neumann}
		Let the domain be a product of three open intervals,
		$\Omega = I_1\times I_2\times I_3 \subset \R^3$. For right-hand
		sides $f \in L^2(\Omega,\C)$ and $g \in L^2(\Omega,\C^3)$, let
		$u \in L^2(\Omega,\C)$ satisfy
		\begin{equation}
			\label{eq:Helmholtz-distributional-Neumann}
			\int\limits_\Omega - u \Delta \phi
			= \int\limits_\Omega f \phi + g \cdot \nabla \phi\qquad
			\forall \phi \in D_\#(\bar\Omega) \text{ with } \partial_3 \phi \in D_\#(\Omega) \,.
		\end{equation}	
		We assume that vertical derivatives of averages vanish along the
		boundaries in the following sense: For every
		$\fhi \in C^{\infty}_\#(I_1 \times I_2) $ there exists a function
		$h = h_\fhi\in L^2(I_3)$ such that
		\begin{equation}
			\label{eq:vanish:distributional:Neumann}
			\int_{I_3} \left\{ \left(\int_{I_1 \times I_2}
			u\, \fhi\right)\, \partial_3 \psi
			+ h \psi\right\} = 0 \qquad
			\forall \psi \in D(I_3)\,.
		\end{equation}	
		Then, the solution $u$ of
		\eqref{eq:Helmholtz-distributional-Dirichlet} is of class
		$u \in H^1_{\#}(\Omega)$ and is a weak solution in the sense that
		\begin{equation}
			\label{eq:IntParts-Neumann}
			\int\limits_\Omega \nabla u \cdot \nabla \phi 
			= \int\limits_\Omega f \phi + g \cdot \nabla \phi\qquad
			\forall \phi \in H^1_{\#}(\Omega) \,.
		\end{equation}
	\end{lemma}
	
	Let us highlight the differences between Lemma \ref
	{lem:regularity:Distrib-Boundary-Dirichlet} and Lemma \ref
	{lem:regularity:Distrib-Boundary-Neumann}. (i) In relation \eqref
	{eq:Helmholtz-distributional-Neumann}, the test-functions do not have
	to vanish at the horizontal boundaries; on the other hand, we demand
	that the vertical derivative vanishes in a neighborhood of the
	boundary. (ii) Relation \eqref {eq:vanish:distributional:Neumann}
	implies that
	$\left(\int_{I_1 \times I_2} u\, \fhi\right)\in H^1(I_3)$, but there
	is no information on the boundary values.
	
	\begin{proof}
		As in the last proof, it suffices to show that a suitable extension
		of $u$ satisfies an equation on an extended domain; an application
		of the lemma of Weyl then yields the $H^1$-regularity up to the
		boundary. The proof is much like the proof of Lemma \ref
		{lem:regularity:Distrib-Boundary-Dirichlet}. We restrict ourselves
		to a sketch of the differences and to the arguments regarding the
		lower boundary.
		
		We assume again $I_3 = (0,l_3)$ and consider the extended domain
		$\check\Omega \coloneqq I_1\times I_2\times (-l_3,l_3)$. In this proof, we
		extend $u$ to an even function and set
		$\check{u}(\hat{x}, x_3)\in L^2(\Omega)$ by
		\begin{equation*}
			\check{u}(\hat{x}, x_3) \coloneqq \begin{cases} u(\hat{x}, x_3) &
				\text{if } x_3 >0\,,
				\\
				u(\hat{x}, -x_3) & \text{if } x_3 <0 \,,
			\end{cases}
		\end{equation*}
		and define the functions $\check{f}$ and $\check{g}$ in the same way
		as even extensions.
		
		Test-functions are chosen with $\varphi \in D_\#(I_1 \times I_2)$
		and $\psi \in C^\infty([-l_3,l_3])$. The symmetrized variant is defined as
		$\check{\psi}(x_3) = \psi(x_3) + \psi(-x_3)$. We note that
		$\check\psi$ is even and continuous across $x_3=0$, the value for
		$x_3=0$ is arbitrary, the derivative vanishes in $x_3=0$. Of course,
		$\del_{x_3} \psi$ is not vanishing in a neighborhood of $x_3=0$. We
		choose an approximation as follows:
		$\check\psi_n \in C^\infty(\overline{I_3})$ with
		$\partial_{x_3} \check\psi_n \in C^\infty_c(I_3)$ such that
		$\check\psi_n \to \check{\psi}|_{I_3}$ in $H^1(I_3)$.
		
		At this point, one performs the same lengthy computation as in the
		proof of Lemma \ref {lem:regularity:Distrib-Boundary-Dirichlet}.
		The first three equalities are identical. The fourth equality
		remains true since $\del_{x_3} \psi$ vanishes for $x_3=0$ (and a
		reference to \eqref{eq:vanish:distributional:Dirichlet} is not
		necessary).
		For the equality that was using \eqref
		{eq:Helmholtz-distributional-Dirichlet}, we now use \eqref
		{eq:Helmholtz-distributional-Neumann}. The result of the
		calculation is
		\begin{equation}
			\label{eq:972345}
			-\int\limits_{\check \Omega} \check{u}\, \Delta \phi = 
			\int\limits_{\check \Omega} \check{f}\, \phi + \check{g} \cdot \nabla \phi
		\end{equation}
		for test-functions of the form $\phi = \fhi(x_1,x_2)\psi(x_3)$. An
		arbitrary function $\phi \in D_\#(\check\Omega)$ can be approximated
		in $H^2(\check\Omega)$ with linear combinations of such functions.
		We obtain that \eqref {eq:972345} holds for all
		$\phi \in D_\#(\check\Omega)$. This concludes the argument for the
		lower boundary $x_3 = 0$, the upper boundary is treated analogously.
		Lemma \ref{lem:Weyl-H1} of Weyl allows to conclude
		$u \in H^1_\#(\Omega)$.
		
		The weak form \eqref {eq:IntParts-Neumann} is obtained from \eqref
		{eq:972345} as follows: An arbitrary $\phi \in H^1_{\#}(\Omega)$ is
		approximated with $\phi_n\in C^\infty_{\#}(\overline\Omega)$ with
		$\del_{x_3} \phi_n = 0$ along horizontal boundaries (approximation
		in $H^1(\Omega)$ as $n\to \infty$ as in Lemma
		\ref{lem:density}). Relation \eqref {eq:972345} is used with
		$\phi = \phi_n$. An integration by parts yields the equality of
		\eqref {eq:IntParts-Neumann} for $\phi_n$. In the limit
		$n\to \infty$, we obtain \eqref {eq:IntParts-Neumann}.
	\end{proof}

\def\cprime{$'$}

\end{document}